\documentclass[11pt]{article}
\usepackage[a4paper, total={6in, 8in}]{geometry}
\usepackage{eurosym}
\usepackage{color}
\usepackage{amsmath}
\usepackage{amsfonts}
\usepackage{amssymb}
\usepackage{authblk}
\usepackage{amsthm}
\usepackage{dsfont}
\usepackage{mathtools}
\usepackage{scalerel,nicefrac}

\newtheorem{theorem}{Theorem}[section]
\theoremstyle{plain}
\newtheorem{lem}[theorem]{Lemma}
\newtheorem{cor}[theorem]{Corollary}
\theoremstyle{definition}
\newtheorem{rem}{Remark}
\newtheorem{dfn}{Definition}

\newcommand{\M}{\mathcal{M}}
\newcommand{\N}{\mathcal{N}}
\newcommand{\F}{\mathcal{F}}
\newcommand{\Ll}{\mathcal{L}}

\newcommand{\I}{\mathrm{I}}
\newcommand{\E}{\mathrm{E}}
\newcommand{\X}{X}
\newcommand{\A}{A}

\newcommand{\mm}{M}
\newcommand{\nn}{N}
\newcommand{\pa}{\mathrm{PA}}

\newcommand{\ssy}{\mathrm{SSy}}
\newcommand{\Th}{\mathrm{Th}}
\newcommand{\sat}{\mathrm{Sat}}
\newcommand{\fix}{\mathrm{Fix}}

\newcommand{\p}{\mathrm{(a)_{\textbf{i}}}}

\begin{document}

\title{Self-embeddings of models of arithmetic;  fixed points, small submodels, and extendability }
\author{Saeideh Bahrami \\

{\small {bahrami.saeideh@gmail.com}}}
\maketitle

\begin{abstract}
	In this paper we will show that for every cut $ I $ of any  countable nonstandard model $ \M $ of $ \I\Sigma_{1} $,  each $ I $-small $ \Sigma_{1} $-elementary submodel  of $ \M $  is of the form of the set of fixed points of some proper initial self-embedding of $ \M $ iff  $ I $ is a strong cut of $ \M $. Especially, this feature  will provide us with some  equivalent conditions with the strongness of the standard cut in a given  countable model $ \M $ of $ \I\Sigma_{1} $. In addition, we will find some criteria for extendability of initial self-embeddings of countable nonstandard models of $ \I\Sigma_{1} $ to larger models.
\end{abstract}	

\section{Introduction}
In 1973, Harvey Friedman proved a striking result for countable nonstandard models of finite set theory, and consequently for countable models of Peano arithmetic, $ \pa $, stating that \textit{every countable nonstandard model of $ \pa $ carries a proper initial self-embedding}; here an \textit{initial self-embedding} is a self-embedding whose image is an initial segment of the ground model \cite{fr}. Afterwards, many versions of Friedman's style Theorem appeared in the literature of model theory of arithmetic (e.g. see \cite{dp} or \cite{w}).  In \cite{our}, it is shown that some results on the set of fixed points of automorphism of countable recursively saturated models of $ \pa $ can be generalized for  initial self-embeddings of countable nonstandard models of $ \I\Sigma_{1} $ (see Theorem 2.4 below). In this paper, inspired by  results about automorphisms of models of $ \pa $, we will investigate some more properties of countable models of $ \I\Sigma_{1} $ through initial self-embeddings.

In \cite{en}, Enayat generalized the notion of a \textit{small} submodel from \cite{las}, to  \textit{$ I$-small}\footnote{In his paper \cite{en}, Enayat called such submodels  $ I$-coded. The name $ I$-small is borrowed from Kossak-Schmerl's  book \cite{ksb}.} for a given cut $ I $  of a model of $ \pa $ (see Definition 1 below), and proved that:
\begin{theorem}[Enayat]
	Suppose $ \M\models\pa $ is countable, recursively saturated, and  $ I $ is a \textit{strong cut} of $ \M $. Moreover, let $ \M_{0}$ be  an $ I$-small elementary submodel of $ \M $. Then there exists some automorphism $ j $ of $ \M $ such that $ \mm_{0} $ is equal to the set of fixed points of $ j $.
\end{theorem}
  In section 3 of this paper, after investigating some basic properties of $ I$-small $ \Sigma_{1}$-elementary submodels  of a countable model $ \M $ of $ \I\Sigma_{1} $ for some cut $ I $ of $ \M $, we will refine the above theorem for initial self-embeddings; i.e  we will show that  $ I $ is strong in $ \M $ iff  every  $ I$-small $ \Sigma_{1}$-elementary submodel of $ \M $ is equal to the set of fixed points of some proper initial self-embedding of $ \M $.  This result also generalizes one of the main theorems of \cite{our} (see Corollary 4.3 below).

Section 4 of this paper, is devoted to the investigation of equivalent conditions to strongness of the standard cut, denoted by $ \mathbb{N} $, in a countable model of $ \I\Sigma_{1} $, through the set of fixed points of  initial self-embeddings. In \cite{ks}, it is shown that:
\begin{theorem}[Kossak-Schmerl]
Suppose $ \M $ is a countable recursively saturated model  of $ \pa $. If $ \mathbb{N} $ is not strong in $ \M $, then for every automorphism $ j $ of $ \M $ the set of fixed points of $ j $ is isomorphic to $ \M $.
\end{theorem}
In Corollary 4.2, we will show that for every countable nonstandard model $ \M $ of $ \I\Sigma_{1} $, if $ \mathbb{N} $ is not strong in $ \M$, then the set of fixed points of any initial self-embedding  $ j $ of $ \M $   is either a model of $ \neg\mathrm{B}\Sigma_{1} $, or is isomorphic to some proper  initial segment  of $ \M $. Then, we conclude that $ \mathbb{N} $ is strong in a countable recursively saturated model $ \M $ of $ \pa $ iff there exists some proper initial self-embedding $ j $ of $ \M $ such that  the set of fixed points of $ j $ is small in $ \M $ and consequently it is not isomorphic to any proper initial segment of $ \M $.

  In section 5, we will study the extendability of initial embeddings of models of $ \I\Sigma_{1} $ to larger models. In particular, we will prove that any isomorphism between two $ \Sigma_{1}$-elementary initial segment  of a countable nonstandard model $ \M $ of $ \I\Sigma_{1} $ is extendable to some initial self-embedding of $ \M $ iff it preserves coded subsets (for the case of automorphisms of countable recursively saturated models of $ \pa $ this condition is only a necessary condition for extendability to larger models \cite{kos-kot}).


\section{Preliminaries}
In this section we will review some definitions and results which are used through this paper. All unexplained notions can be found in \cite{hp} and \cite{kaye}.
\begin{itemize}
\item  Through this paper, we will work in the language of arithmetic $ {\mathcal{L}_{ A}:=\{+,.,<,0,1\}} $. For a given class $ \Gamma $ of $ \mathcal{L}$-formulas (where $ \mathcal{L}\supseteq\mathcal{L}_{ A} $), $ \I\Gamma $ is the fragment of $\mathrm{PA}^{*}:=\pa(\mathcal{L})$ with the induction scheme
limited to formulas of $ \Gamma $. The \textit{$\Gamma$-Collection} scheme, denoted by $
\mathrm{B}\Gamma$, consists of the 
formulas of the following form for every $\varphi \in \Gamma$:
\begin{center}
$\forall \bar{z},u \ ((\forall x<u \ \exists y \ \varphi(x,y,\bar{z}))\rightarrow \exists v \ (\forall x<u \ \exists y<v \ \varphi(x,y,\bar{z})) ).$
\end{center}
Moreover, the \textit{strong $\Gamma$-Collection} scheme, denoted by $
\mathrm{B}^{+}\Gamma$, consists of the 
formulas of the following form for every $\varphi \in \Gamma$:
\begin{center}
	$\forall \bar{z},u \ \exists v \ \forall x<u \ (\exists y \ \varphi(x,y,\bar{z})\rightarrow \exists y<v  \ \varphi(x,y,\bar{z})).$
\end{center}
It is folklore that   $ \I\Sigma_{n+1}\vdash\mathrm{B}^{+}\Sigma_{n+1}\vdash\mathrm{B}\Sigma_{n+1} $ for all $ n\in\omega $; moreover, for every $ n\in\omega $, neither $ \I\Sigma_{n}\nvdash\mathrm{B}\Sigma_{n+1} $, nor $ \I\Sigma_{n}\nvdash\neg\mathrm{B}\Sigma_{n+1} $ (see \cite[Ch. I]{hp}).
\item Within $ \I\Delta_{0}+\mathrm{Exp}$, the $ \Delta_{0}$-formula $ x\E y $ denotes the \textit{Ackermann's membership relation}, asserting that \textit{"the $ x$-th bit of the binary expansion of $ y $ is 1"}. For every $ {\M\models\I\Delta_{0}+\mathrm{Exp}} $ and each $ a\in\mm $, $ a_{\E} $ denotes the set of $ \E$-members of $ a $ in $ \M $. Moreover, the  $ \Delta_{0}$-formulas  $ \mathrm{Card}(x)=y $, $ \langle\bar{x}\rangle=y $, $ \mathrm{Len}(x)=y $,  $ (x)_{y}=z $, and $ x\upharpoonright_{y}=z $ respectively express that \textit{"there exists some bijection between $  y $ and the set coded by $ x $"}, \textit{"the sequence number of $ \bar{x} $ is $ y $"},  \textit{"length of the sequence coded by $ x $ is $ y $"}, \textit{"the $ y$-th element of the sequence number $ x $ is $ z $"}, and \textit{"the restriction of the sequence number $ x $ to $ y $ is $ z $"}.  In addition, for every formula $ \varphi(x) $, by the formula $ y=\mu_{x}\ \varphi(x) $ we mean \textit{"y is the least element such that $ \varphi(y) $ holds"}.

Furthermore, for every $ n\in\omega $ there exist $ \mathcal{L}_{ A}$-formulas $ \sat_{\Sigma_{n}} $ and $ \sat_{\Pi_{n}} $ which define the satisfaction predicate for $\Sigma_{n}$-formulas and $ \Pi_{n}$-formulas respectively, in an ambient model. For every natural number $ n>0 $, it can be shown that $ \sat_{\Sigma_{n}} $ and $ \sat_{\Pi_{n}} $ are $ \Sigma_{n} $ and $ \Pi_{n} $ respectively in $ \I\Sigma_{1} $. Moreover, $ \sat_{\Delta_{0}}\in\Delta_{1}^{\I\Sigma_{1}} $ \cite[ch. I, Thm. 1.75]{hp}. If $ \M $ is a nonstandard model of $ \I\Sigma_{n} $, the aforementioned feature along  with $ \Sigma_{n}$-Overspill in $ \M $ imply that every coded $ \Sigma_{n}$-type and every coded bounded $ \Pi_{n}$-type is realized in $ \M $.
\item \textbf{$ \Sigma_{n}$-Pigeonhole Principle.} For every $ n>0 $, if $ \M\models\I\Sigma_{n} $, $ a\in\mm $, and $ \varphi $ is a $ \Sigma_{1}$-formula which defines a function from $ a+1 $ into $ a $ in $ \M $, then $ \varphi $ is not one-to-one \cite{hp}.
\item Given $ \Ll_{\A}$-structure $ \M $ and subset $ \X $ of $ \mm $, for every  $ n>0 $, we define:
\begin{center}
\begin{itemize}
\item $ \mathrm{K}^{n}(\M;\X):=$the set of all $ \Sigma_{n}$-definable element of $ \M $ with parameters from $ \X $;
\item $ \mathrm{I}^{n}(\M;\X):=\{x: \ x\leq a \text{ for some } a\in\mathrm{K}^{n}(\M;\X)\} $;
\item $ \mathrm{H}^{n}(\M;\X):=\bigcup_{k\in\omega}\mathrm{H}_{k}^{n}(\M;\X) $, where:
 \begin{center}$ \mathrm{H}_{0}^{n}(\M;\X):=\mathrm{I}^{n}(\M;\X),$ and \\ $ \mathrm{H}_{k+1}^{n}(\M;\X):=\mathrm{I}^{n}(\M;\mathrm{H}_{k}^{n}(\M;\X))$.
 	\end{center}
\item $ {\mathrm{K}(\M;\X):=\cup_{n\in \omega}\mathrm{K}^{n}(\M;\X)  }$.
\end{itemize}
\end{center}
(When $ X=\emptyset $, we omit $ X $ from the notations.) Clearly, $\mathrm{I}^{n}(\M;\X)  $ and $ \mathrm{H}^{n}(\M;\X) $ are initial segments of $ \M $. The following properties of these submodels of $ \M $ are well-known (e.g. see  \cite[Ch. IV, Thm. 1.33]{hp}):
\begin{theorem}
Suppose $ n>0 $, and $ \M\models\I\Sigma_{n} $ and $ X\subseteq M $, then the following hold:
\begin{itemize}
\item[(1)] $\mathrm{K}^{n}(\M;\X)\prec_{\Sigma_{n}}\M  $, and if $ \mathrm{K}^{n}(\M) $ is nonstandard, then $\mathrm{K}^{n}(\M)\models\neg\mathrm{B}\Sigma_{n}  $. 
\item[(2)] $\mathrm{I}^{n}(\M;\X)\prec_{\Sigma_{n-1}}\M  $, and  if $ \mathrm{K}^{n}(\M;\X) $ is nonstandard, then $\mathrm{I}^{n}(\M;\X)\models\mathrm{B}\Sigma_{n}  $. 
\item[(3)] $\mathrm{H}^{n}(\M;\X)\prec_{\Sigma_{n}}\M  $  , and if $ \mathrm{K}^{n}(\M;\X) $ is nonstandard, then $\mathrm{H}^{n}(\M;\X)\models\mathrm{B}\Sigma_{n+1}  $. 
\end{itemize}
\end{theorem}
\item A given structure $ \M $ is called  \textit{recursively saturated} if it realizes every recursive type with finite parameters in $ M $. In \cite{bs}, Barwise and Shilipf showed that  \textit{any countable  model $ \M $ of $ \pa $ is recursively saturated iff it carries an \textit{inductive satisfaction class}}; here  an \textit{inductive satisfaction class} $  S $ of $ \M $ is a subset of $ \mm $ which contains $ \langle\varphi,a\rangle $ such that (1) $ \M\models \mathrm{Form}(\varphi) $, (2) $ (\M; S)\models\pa^{*} $, and (3) $ (\M; S) $ satisfies Tarski's inductive conditions for satisfaction (for a more precise definition see \cite{kaye}). It is folklore that for every countable recursively saturated model $ \M $ of $ \pa $ there exists some inductive satisfaction class $  S $ such that $ (\M; S) $ is also recursively saturated (e.g. see \cite{kos3}).
\item For every cut $ I $ of $ \M $ the \textit{$ I$-Standard System} of $ \M $, denoted by $ \ssy_{I}(\M) $, is the family of subsets of $ I $ of the form $ I\cap a_{\E} $ for some $ a\in\mm $. By $ \ssy(\M) $ we mean $ \ssy_{\mathbb{N}}(\M) $. It is well-known that for every model $ \M $ of $ \I\Sigma_{n} $ (for  $ n>0 $), $ \ssy_{I}(\M) $ is equal to the family of subsets of $ I $ which are $ \Sigma_{n}$-definable (with parameters) in $ \M $ (see \cite[Ch. I]{hp}). Moreover, it is easy to check that if $ \N$ is an initial segment and a submodel of $ \M $ containing $ I $, then $\ssy_{I}(\M)=\ssy_{I}(\N)  $ (see \cite{kaye}).
\item A given model $ \M $ of $ \I\Delta_{0} $ is called \textit{1-tall} if $ \mathrm{K}^{1}(\M;a) $ is  cofinal in $ \M $ for no $ a\in\mm $; and it is called \textit{1-extendable} if it possesses some end extension $ \N\models\I\Delta_{0} $ such that $ \Th_{\Sigma_{1}}(\M)=\Th_{\Sigma_{1}}(\N) $. Dimitracopoulos and Paris, in \cite{dp} showed that:
\begin{theorem}[Dimitracopoulos-Paris]
\begin{itemize}
\item[(1)] For any two countable and nonstandard models $ \M $ and $ \N $ of $ \I\Delta_{0}+\mathrm{Exp} $ such that $ \M $ is 1-extendable and $ \N $  is 1-tall, there exists a proper initial embedding from $ \M $ into $ \N $ iff $ \ssy(\M)=\ssy(\N) $ and $ \Th_{\Sigma_{1}}(\M)\subseteq\Th_{\Sigma_{1}}(\N) $.
\item[(2)] Any 1-tall countable model $ \M $ of $ \mathrm{B}\Sigma_{1}+\mathrm{Exp} $ in which $ \mathbb{N} $ is not $ \Pi_{1} $-definable (without parameters), is 1-extendable.
\end{itemize}
\end{theorem}
\item A given  cut $ I $ of a  model $ \M $ is called \textit{strong} if for every coded function $ f $ of $ \M $ whose domain contains $ I $, there exists some $ e>I $ such that $ f(i)\in I $ iff $ f(i)<e $ for all $ i\in I $. Paris and Kirby, in  \cite{kp}, proved that \textit{ $ I $ is a strong cut of a model $ \M $ of $ \I\Delta_{0}+\mathrm{Exp} $ iff  $ (I,\ssy_{I}(\M))\models\mathrm{ACA}_{0} $} (here $ \mathrm{ACA}_{0} $ is the subsystem of second order arithmetic with the comprehension scheme  restricted to formulas with no second order quantifier).   
\item For given $ \Ll_{\A}$-structures $ \M $ and $ \N $,  an  \textit{(a proper) initial embedding} $ j $ is an embedding from $ \M $ into $ \N $ whose image is an (a proper) initial segment of $ \N $. To every self-embedding $ j $ of $ \M $, we associate two subsets of $ \mm $:
\begin{center}
	$\mathrm{I}_{\mathrm{fix}}(j):=\{m\in M:\forall x\leq m\ j(x)=x\},$ and\\[0pt
	]
	$\mathrm{Fix}(j):=\{m\in M:j(m)=m\}.$
\end{center}
In \cite{our}, it is shown that for every model $ \M $ of $ \I\Sigma_{1} $, and any self-embedding $ j $ of $ \M $, it holds that $K^{1}(\mathcal{M})\prec _{\Sigma _{1}}
\mathrm{Fix}(j)\prec_{\Sigma _{1}}\mathcal{M}$. Consequently,  $ \mathrm{Fix}(j)\models \I\Delta_{0}+\mathrm{Exp}. $ The following results on the set of fixed points of initial self-embeddings were also proved in \cite{our}:
\begin{theorem}[B-Enayat]
Let $ \M $ and $ \N $ be countable nonstandard models of $ \I\Sigma_{1} $, $ c\in\mm $ and $d,b\in\nn  $, and $ I $ be a proper cut shared by $ \M $ and $ \N $ which is closed under exponentiation. Then the following are equivalent:
\begin{itemize}
\item[(1)] There exists some proper  initial embedding $ j $ from $ \M $ into $ \N $ such that $ I\subseteq\I_{\mathrm{fix}}(j) $, $ j(\mm)<b $, and $ j(c)=d $.
\item[(2)] $ \ssy_{I}(\M)=\ssy_{I}(\N) $, and for every $ \Delta_{0}$-formula $ \delta(z,x,y) $ and every $ i\in I $ it holds that:
$$ \M\models\exists z \ \delta(z,c,i) \ \Rightarrow \ \N\models\exists z<b \ \delta(z,d,i) . $$
\end{itemize} 
\end{theorem}
\begin{rem} With the above assumptions, suppose    $ a\in\mm\cap\nn $ such that
	 for all $ \Delta_{0}$-formula $ \delta $ and for every $ i\in I $ it holds that: 
	\begin{center}
		$ \M\models\exists z \ \delta(z,c,(a)_{i}) \ \Rightarrow \ \N\models\exists z<b \ \delta(z,d,(a)_{i}) $. 
	\end{center}
	Then, by an appropriate modification in the proof of Theorem 2.3,   we can manage to construct the above proper initial embedding  $ j $  with the additinal feature that  $ {j((a)_{i})=(a)_{i}} $ for every $ i\in I $. 
\end{rem}
\begin{theorem}[B-Enayat]
Suppose $ \M\models\I\Sigma_{1} $ is countable and nonstandard and $ I $ is a cut of $ \M $. Then the following hold:
\begin{itemize}
\item[(1)] $ I $ is closed under exponentiation iff there exists some proper initial self-embedding $ j $ of $ \M $ such that $ \I_{\mathrm{fix}}(j)=I $.
\item[(2)] $ I $ is strong in $ \M $ and $ I\prec_{\Sigma_{1}}\M $, iff there exists some proper initial self-embedding $ j $ of $ \M $ such that $ \fix(j)=I $.
\item[(3)]  $ \mathbb{N} $ is strong in $ \M $  iff there exists some proper initial self-embedding $ j $ of $ \M $ such that $ \fix(j)=\mathrm{K}^{1}(\M) $.
\end{itemize}
\end{theorem}

\item The following lemma from \cite{our} will be useful in section 4 of this paper:
\begin{lem}
Suppose $ \M\models\I\Delta_{0}+\mathrm{Exp} $ in which $\mathbb{N}$ is not a strong cut, then for any self-embedding $j$ of $\mathcal{M}$, the following hold:
\begin{itemize}
\item[(1)] 	The nonstandard fixed points of $j$ are downward cofinal in the
nonstandard part of $\mathcal{M}$.
\item[(2)]  
 For every element $a\in \mm$, and $ m\in\fix(j) $ there exists an element $b\in \mathrm{Fix}
(j)$ such that: 
\begin{center}
$\mathrm{Th}_{\Sigma _{1}}(\mathcal{M};a,m)\subseteq \mathrm{Th}_{\Sigma _{1}}(\mathcal{M};b,m).$
\end{center}
\end{itemize}
 \end{lem}

\item \textbf{Convention.} Suppose $ \M\models\I\Sigma_{1} $ and $ \langle\delta_{r}: \ r\in M\rangle $ is  a canonical enumeration of $ \Delta_{0}$-formulas in $ \mathcal{M} $. For every $ r\in M $: 
\begin{itemize}
	\item   $ f_{r}(\diamondsuit)=\blacklozenge$ denotes the following partial $ \Sigma_{1}$-function in $ \M $:  $$ \exists z ((z)_{0}=\blacklozenge \ \wedge \ z=\mu_{y}\mathrm{Sat}_{\Delta_{0}}(\delta_{r}(\diamondsuit,(y)_{0},(y)_{1})). $$
	\item The notation $ [f_{r}(\bar{x})\downarrow] $  denotes the $ \Sigma_{1}$-formula $  \exists z,y \ \mathrm{Sat}_{\Delta_{0}}(\delta_{r}(\bar{x},y,z)) $, and ${ [f_{r}(\bar{x})\downarrow]^{<w}} $  stands for the formula $  \exists z,y<w \ \mathrm{Sat}_{\Delta_{0}}(\delta_{r}(\bar{x},y,z)) $. 
\end{itemize}
Finally, we put $ \mathcal{F}(\M) $ to be the collection of all $ \emptyset$-definable partial $ \Sigma_{1}$-functions in $ \mathcal{M} $. As noted in \cite{our}, if $ \M $ and $ \N $ are two models of $ \I\Delta_{0} $ such that $ \Th_{\Sigma_{1}}(\M)=\Th_{\Sigma_{1}}(\N) $, then  $\F(\M)=\F(\N)= \F:=\{f_{n}: n\in\mathbb{N}\} $. Moreover, in \cite{our} it is shown that: $$ \mathrm{K}^{1}(\M;a)=\{f(a):  \ f\in\F \text{ and } \M\models[f(a)\downarrow] \} .$$
\end{itemize}

\section{$ I $-small $ \Sigma_{1}$-elementary submodels}
In \cite{las}, Lascar introduced a class of submodels of models of arithmetic, namely \textit{small} submodels, which resemble those submodels of a  model of set theory whose cardinality is less than the cardinality of the ground model. Then,  Enayat inspired by a result of Schmerl (stated without proof as Theorem 5.7 in \cite{kkk}), generalized this notion in \cite{en}. In this section we will  prove some results about these submodels. 
\begin{dfn}
 For a given proper cut $ I $  of a model $ \M $ of $ \I\Delta_{0}+\mathrm{Exp} $,  subset $ \X $ of $ \mm $ is called \textit{$ I$-small} in $ \M $  if there exists some $ a\in\mm $ such that $ \X=\{(a)_{i} : \  i\in I\} $, and $ (a)_{i}\neq(a)_{j} $ for all distinct $ i,j\in I $. When $ I=\mathbb{N} $, we simply use  \textit{small} for $ \mathbb{N}$-small.
\end{dfn}
 It is easy to see that for every model $ \M $ of $ \I\Sigma_{1} $, each proper cut $ I $ of $ \M $  is $ I$-small. Moreover, for every $ a\in\mm $, $ \mathrm{K}^{1}(\M;a) $ is small in $ \M $. In \cite{ks}, it is shown that  every recursively saturated model $ \M$ of $\pa $ possesses some small submodel which is not finitely generated. This  result can be generalized for $ I$-small submodels, when $ I $ is a strong cut of $ \M $ (see Theorem 3.2 below). Furthermore, By using compactness arguments, for every model $ \M $ of $ \I\Sigma_{1} $, we can find some elementary extension of $ \M $ in which it is small. And finally, in \cite{kos-kot96} it is shown that every nonstandard small submodel is a \textit{mixed} submodel (i.e. neither cofinal, nor initial segment). In a similar manner,   for every cut $ I $ of a model $ \M $ of $ \I\Sigma_{1} $, and each $ I$-small  submodel $ \M_{0} $ of $ \M $, if $ I\subsetneq M_{0} $ then $ M_{0} $ is mixed in $ \M $ (since if $ \mm_{0}:=\{(a)_{i}: \ i\in I\} $, and $ \A:=\{i\in I: \ \M\models\neg i\E(a)_{i} \} $, then $\A\in\ssy_{I}(\M)\setminus \ssy_{I}(\M_{0}) $. So $ \M_{0} $ cannot be an initial segment  of $ \M $).
 
 In the following lemma we will show that in the definition of $ I $-small, if $ I $ is a strong cut or it is equal to $ \mathbb{N} $, then  the condition $ (a)_{i}\neq(a)_{j} $ for all distinct $ i,j\in I $
 , can be eliminated:
\begin{lem}
	Suppose $ \mathcal{M}\models\mathrm{I}\Sigma_{1} $ is nonstandard, $ I\subsetneq_{e}\mathcal{M} $, $ \M_{0} $ is a submodel of $ \M $ such that $ \mm_{0}=\lbrace (a)_{i}:\ i\in I\rbrace $ for some  $ a\in M $. Then the following hold:
	\begin{itemize}
		\item[(1)] If $ I=\mathbb{N} $, then $ \M_{0} $ is small.
		\item[(2)] If $ I $ is strong in $ \mathcal{M} $, then $ \M_{0}$ is $ I$-small.
			\end{itemize}
\end{lem}
\begin{proof}
	First, we will inductively define the following $ \Delta_{0}$-function (with parameters) in $ \mathcal{M} $:
	\begin{center}
		$ g(0):=(a)_{0} $,  \\and\\
		$ g(x+1):=y $ iff $ \exists r<\mathrm{Len}(a) \ \left(\begin{array}{c}
		y=(a)_{r} \ \wedge \\  r=\mu_{z}\left(\begin{array}{c} \exists u<z
	\left(\begin{array}{c}g(x)=(a)_{u} \ \wedge \\ \forall w<z \ ( (a)_{w}\neq (a)_{z}  \wedge \exists v\leq u ((a)_{w}=(a)_{v}) )\end{array}\right)
	\end{array}\right)	\end{array}\right)
		  $.
	\end{center}
	Note that by the way we defined $ g $, its domain is an initial segment of $ \M $, and $ {\mathrm{Dom}(g)\leq\mathrm{Len}(a)} $. Moreover, since $ I $ and $ \mm_{0} $ are not $ \Delta_{0}$-definable in $ \M $, then  $ I\subsetneq\mathrm{Dom}(g) $. So  by $ \Sigma_{1}$-induction in $ \M $, we can find some $ d\in M $ such that $ (d)_{i}=g(i) $ for every $ i\in I $. Clearly, $ (d)_{i}\neq (d)_{j} $ for every distinct $ i,j\in I $, and $ {\mm_{0}\subseteq\lbrace (d)_{i}: \ i\in I\rbrace } $.  Now, in each case of the statement of theorem  we will prove that $ \lbrace (d)_{i}: \ i\in I\rbrace \subseteq\mm_{0} $: 
	\begin{itemize}
		\item[(1)] Suppose $ I=\mathbb{N} $.  If  $ \lbrace (d)_{n}: \ n\in\mathbb{N}\rbrace\nsubseteq\mm_{0} $, then  there exists the least number  $ n\in\mathbb{N} $ such that $ (d)_{n}\notin\mm_{0} $. So by the definition of $ g $, there exist some $ m\in\mathbb{N} $ and some $ r\in M\setminus\mathbb{N} $ such that ${ (d)_{n-1}=(a)_{m}} $ and $ (d)_{n}=(a)_{r} $. Therefore, by the definition of $ g $, it holds that $  \mm_{0}=\lbrace(a)_{0},...,(a)_{m}\rbrace$, which is a contradiction.
		\item[(2)] In the general case with the extra assumption that $ I $ is strong in $ \mathcal{M} $, consider the following partial $ \Delta_{0}$-function in $ \M $: 
		\begin{center}
			$h(x):=\mu_{r}( (d)_{x}=(a)_{r})$.
		\end{center}
		Since $ I $ is strong and $ I\subseteq\mathrm{dom}(h) $ (because $ g $ is well-defined on $ I $), there exists some $ e\in M $ such that $ h(i)\in I$ iff $ h(i)<e $, for all $ i\in I $. Moreover, by the definition of $ d$, $g $ and $ h $, for every $ i\in I $ it holds that $ (d)_{i}=(a)_{h(i)} $. So it suffices to prove that $ h(i)<e $ for every $ i\in I $. Suppose not;   so there exists some $ i_{0}\in I $ which is  the least element of $  M $ such that $ h(i_{0})>e $.  Now, by the way we defined $ g $ and $ h $, it holds that:\\
		
		$ \mathcal{M}\models\forall i<h(i_{0}) \ ((a)_{i}\neq(a)_{h(i_{0})}\wedge \exists j\leq h(i_{0}-1) ((a)_{i}=(a)_{j})) $. \\
		
		Therefore, $ \mm_{0}=\lbrace x\in M: \ \mathcal{M}\models\exists i\leq h(i_{0}-1) \ x=(a)_{i} \rbrace $. So $ \mm_{0} $  is $  \Delta_{0}$-definable in $ \mathcal{M} $, which is a contradiction.
	\end{itemize}
\end{proof}

In the following theorem, we will show that when $ I $ is strong, the basic properties which hold for small submodels, also hold for $ I$-small ones.

\begin{theorem}
	Let $ \mathcal{M}\models\mathrm{I}\Sigma_{1} $ be nonstandard, and $ I $ be a strong cut of $ \mathcal{M} $. Then:
	\begin{itemize}
	\item[(1)] For every $ a\in M $, $ \mathrm{K}^{1}(\mathcal{M};I\cup\lbrace a\rbrace) $ is $  I$-small.
	\item[(2)] If $ \mathcal{M}_{0} $ is an $  I$-small submodel of $ \mathcal{M} $, then $  I\subseteq M_{0} $.
	\item[(3)] If $ \M\models\pa $ is countable and recursively saturated, then there exists some $  I$-small elementary submodel of $ \mathcal{M} $ which is not of the form of $ \mathrm{K}(\mathcal{M}; I\cup\lbrace a\rbrace) $ for any $ a\in M $.
	\end{itemize}
\end{theorem}
\begin{proof}
\begin{itemize}
\item[(1)] First fix some
 arbitrary $ s>I $. So by using  strong  $ \Sigma_{1}$-Collection in $ \mathcal{M} $ for the formula $ \sat_{\Delta_{0}}(\delta_{r}(i,a,z)) $, we will find some $ b\in M $ such that:
\begin{center}
	$ \mathcal{M}\models\forall \langle r,i\rangle<s \ ([f_{r}(i,a)\downarrow]\rightarrow[f_{r}(i,a)\downarrow]^{<b}) $.
\end{center}  
Then, by using  $ \Sigma_{1}$-induction we observe that $ \M\models\exists y \ \forall \langle r,i\rangle<s \  \varphi(y,r,i,a,b) $, in which  $ \varphi(y,r,i,a,b) $ is the following $ \Delta_{0}$-formula:
\begin{center}
	$  \left( \begin{array}{c}
([f_{r}(i,a)\downarrow]^{<b}\rightarrow(y)_{ \langle r,i\rangle}=f_{r}(i,a)) \ \wedge \
 (\neg[f_{r}(i,a)\downarrow]^{<b}\rightarrow(y)_{ \langle r,i\rangle}=0)
	\end{array}\right)   $.
\end{center}
As a result, if $ d\in\mm $ is  such that $ \M\models  \forall \langle r,i\rangle<s \  \varphi(d,r,i,a,b) $, then:
$$ \mathrm{K}^{1}(\mathcal{M}; I\cup\lbrace a\rbrace)=\{(d)_{i}: \ i\in I \}. $$ So by Lemma 3.1,   $ \mathrm{K}^{1}(\mathcal{M}; I\cup\lbrace a\rbrace) $ is $  I$-small in $ \mathcal{M} $.
\item[(2)] The exact argument used  in \cite[Thm. 4.5.1]{en} works here: 	let $ { M_{0}=\lbrace (a)_{i}: \ i\in I\rbrace} $ for some $ a\in M $ such that $ (a)_{i}\neq(a)_{j} $ for all distinct $ i, j\in I $. Then put:
$${ Z:=\lbrace \langle y,z\rangle\in M: \ \mathcal{M}\models (a)_{y}=z \rbrace }.$$
Since $ Z $ is $ \Delta_{0}$-definable in $ \mathcal{M} $, then $ X:= I\cap Z \in \mathrm{SSy}_{ I}(\mathcal{M}) $. As a result, because $  I $ is strong in $ \mathcal{M} $, $ ( I, X)\models\mathrm{PA}^{*} $.
Now, suppose $  I\nsubseteq M_{0} $. So $ ( I, X)\models\exists x \ (\forall y \ \langle y,x\rangle\notin X) $. Let $ ( I, X)\models \textbf{x}_{0}:=\mu_{x} ( \forall y \ \langle y,x\rangle\notin X)  $. Therefore, $ \textbf{x}_{0}\notin M_{0} $. So since $ \textbf{x}_{0}\neq 0 $, and by the definition of $ \textbf{x}_{0} $, we conclude that $ \textbf{x}_{0}-1\in M_{0} $, which contradicts the fact that $ \mathcal{M}_{0} $ is a submodel of $ \mathcal{M} $.
	\item[(3)] We will generalize the method used in \cite[Pro. 2.10]{ks}: let $  S $ be a  nonstandard inductive satisfaction class for $ \mathcal{M} $ such that $ (\M; S) $ is recursively saturated.  Put $ \mathcal{M}^{*}:=(\mathcal{M}; S) $, and $ \mathcal{N}:=\mathrm{K}(\mathcal{M}^{*};I\cup\{s\}) $ for some $ s>I $. First, note that $ \nn $ is $ I$-small in $ \M $: since $ \M^{*} $ is a countable recursively saturated model of $  \pa^{*}$, so it also possesses an inductive satisfaction class. Moreover,   $ I $ is also  strong in $ \M^{*} $. Therefore,  by repeating the argument used in the proof of part (1) of this theorem, and Lemma 3.1(2), we can show that  $ \nn $ is  $ I$-small in $ \mathcal{M} $. \\
Moreover, on one hand, it is easy to see that $  S\cap\nn $ is a  nonstandard satisfaction class for the $ \Ll_{\A}$-structure $ \N $. So $ \N $ is also a  recursively saturated model of $ \pa $. On the other hand, $ I $ is a proper initial segment of $ \N $ (because $ s>I $). Therefore, $ \N $ is of the form of $ \mathrm{K}(\M;I\cup\{a\}) $ for no $ a\in\mm $. 

\end{itemize}
\end{proof}

The  following lemma will be useful in the proof of the main theorem of this section:

\begin{lem}
	Suppose $ \mathcal{M}\models\mathrm{I}\Sigma_{1} $,  $  I $ is a strong cut of $ \mathcal{M} $, and  $ a\in\mm\setminus I $  such that $ (a)_{i}\neq(a)_{j} $ for all distinct $ i,j\in I $. Moreover, let   $ { M_{0}=\lbrace (a)_{i}: \ i\in I\rbrace} $ be a $ \Sigma_{1} $-elementary submodel of $ \M $,  $  X\subseteq M_{0}$ be coded in $ \mathcal{M} $, and  $ i_{0}\in I $ such that $ i<i_{0}$ for all $ (a)_{i}\in\X $. Then $  X $ is coded in $ \mathcal{M}_{0} $.
\end{lem}
\begin{proof}
Suppose $ \alpha\in M $ codes $  X $ in $ \M $. So $ \mathcal{M}\models \overset{\delta(\alpha,\sigma,i_{1})}{\overbrace{\alpha=\sum_{i<i_{1}}2^{(\sigma)_{i}}} }$, in which $ i_{1}=\mathrm{Card}( X)\leq i_{0} $ and $ \sigma:=\langle x: \ x\mathrm{E}\alpha\rangle $ (so $ \mathrm{Len}(\sigma)=i_{1} $). Since $ \delta(x,y,z) $ is a $ \Delta_{0}$-formula and $ \mathcal{M}_{0}\prec_{\Sigma_{1}}\mathcal{M} $, it suffices to prove that $ \sigma\in M_{0} $. For this purpose let $ Y:=\{i<i_{0}:\ \mathcal{M}\models (a)_{i}\E\alpha\} $. Then there exists some $ \gamma\in I $ which codes $ Y $.\\ Now, we define:
	 \begin{equation*}
	h(z) :=
	\begin{cases}
	\mu_{u}(\langle (a)_{x}: \ x\mathrm{E}z \rangle=(a)_{u}\wedge u<\mathrm{Len}(a)) & \text{if $\mathcal{M}\models\exists u<\mathrm{Len}(a) \ \langle (a)_{x}: \ x\mathrm{E}z \rangle=(a)_{u}$;
	}
	\\
	0 & \text{otherwise}
	\end{cases}
	\end{equation*} 
Since 	$  I $ is  strong in $ \M $, there exists some $ e $ such that $ h(i)>e $ iff $ h(i)>I $, for all $ i\in I $.  We claim that  $ \mathcal{M}\models \forall x \ \varphi(x,a,\gamma,e) $,  where $\varphi(x,a,\gamma,e)  $ is the following $ \Delta_{0}$-formula:
	\begin{center}
	$ \forall y<\mathrm{Len}(x) \ \exists z\mathrm{E}\gamma \ ((x)_{y}=(a)_{z})\rightarrow\exists w<\min\{e,\mathrm{Len}(a)\} \ (x=(a)_{w}) $.
	\end{center}
Therefore, $ \M\models\varphi(\sigma,a,\gamma,e) $, which implies that  $ \sigma=(a)_{c} $ for some $ c<\min\{e,\mathrm{Len}(a)\} $. So $ \sigma=(a)_{h(\gamma)} $ and $ h(\gamma)<e $, which implies that $ \sigma\in M_{_{0}} $.\\
In order to prove the above claim, we will use $ \Delta_{0}$-induction inside $ \M $: let $ \textbf{x}\in M $ such that  $ {\mathcal{M}\models\varphi(w,a,\gamma,e)} $ for every $ w<\textbf{x} $, and $ \mathcal{M}\models\forall y<\mathrm{Len}(\textbf{x}) \ \exists z\mathrm{E}\gamma \ ((\textbf{x})_{y}=(a)_{z}) $. So by induction hypothesis ${\mathcal{M}\models \textbf{x}\upharpoonright_{\mathrm{Len}(\textbf{x})-1}=(a)_{\textbf{z}} }$ for some $ \textbf{z}<\min\{e,\mathrm{Len}(a)\} $. Then, we put $ Z:=\{i<\gamma: \ \mathcal{M}\models\exists y<\mathrm{Len}(\textbf{x})-1 \ (\textbf{x})_{y}=(a)_{i}\} $, and let $ \textbf{z}_{0}\in I $ code $ Z $. As a result, $ h(\textbf{z}_{0})\leq \textbf{z}<\min\{e,\mathrm{Len}(a)\} $, which implies that $ \textbf{x}\upharpoonright_{\mathrm{Len}(\textbf{x})-1}=(a)_{h(\textbf{z}_{0})}\in M_{0} $. So since $ \mathcal{M}_{0}\prec_{\Sigma_{1}}\mathcal{M} $, then $ \textbf{x} $ is in $  M_{0} $. Therefore, $ \textbf{x}=(a)_{i} $ for some $ i\in I<\min\{e,\mathrm{Len}(a)\} $.
\end{proof}

Now we are ready to prove the main theorem and corollary of this section. The method we use for proving   Theorem 3.4 is a  a combination of the back-and-forth method used in \cite[Thm. 6.1]{our} and \cite[Thm. 5.6]{kkk}. 

\begin{theorem}
	Assume $ \mathcal{N}\models \mathrm{I}\Sigma_{1} $ is countable and nonstandard, $  I $ is a strong cut of $ \mathcal{N} $, and $ \mathcal{N}_{0} $ is an $  I$-small $ \Sigma_{1} $-elementary submodel of  $ \N $ such that $ I\neq\nn_{0} $. Then there exists some proper initial self-embedding $ j $ of $ \mathrm{H}^{1}(\N;\nn_{0})  $ such that $ { N_{0}=\mathrm{Fix}(j)} $.

\end{theorem}
\begin{proof}
Put $ \M:=\mathrm{H}^{1}(\N;\nn_{0})  $. So by Theorem 2.1, $ \M $ is a $ \Sigma_{1}$-elementary initial segment of $ \N $ such that $ \M\models\I\Sigma_{1} $, and it is easy to see that $ I $ is also strong in $ \M $. Moreover,  since $ N_{0}\neq I $, by using $ \Sigma_{1} $-Overspill in $ \M $ we can find some $ a\in M $ such that $ N_{0}=\{(a)_{i}: \ i\in I\} $ and $ (a)_{i}\neq(a)_{i} $ for distinct $ i,j\in I $. In order to construct $ j $, first by using strong $ \Sigma_{1}$-Collection in $ \mathcal{M} $, we will find some $ b\in M $ such that:
\begin{center}
	$ \mathcal{M}\models [f((a)_{i})\downarrow]\rightarrow[f((a)_{i})\downarrow]^{<b} $, for all $ f\in\mathcal{F} $ and all $ i\in I $.
\end{center}

Then, by using back-and-forth method we will inductively build finite functions   $\bar{u}\mapsto\bar{v}$  such that $ \bar{u},\bar{v}\in\mm $, and $\mathcal{M}\models (\bar{v}<b \ \wedge \ \mathrm{P}(\bar{u},\bar{v}) \ \wedge \mathrm{Q}(\bar{u},\bar{v}) )$, in which: 
	\begin{center}
	$ \mathrm{P}(\bar{u},\bar{v})\equiv  [f(\bar{u},(a)_{i})\downarrow]\rightarrow[f(\bar{v},(a)_{i})\downarrow]^{<b} $, for all $f\in\mathcal{F} $ and $ i\in I $;\\ and\\
	$ \mathrm{Q}(\bar{u},\bar{v})\equiv \left(\begin{array}{c}
[f(\bar{u},(a)_{i})\downarrow]\wedge [f(\bar{v},(a)_{i})\downarrow]^{<b} \wedge\\ f(\bar{u},(a)_{i})\notin N_{0}
	\end{array}\right)   \Rightarrow{f(\bar{u},(a)_{i})\neq f(\bar{v},(a)_{i})} $, for all $ f\in\mathcal{F}$ and all $ i\in I $.
\end{center}
Through the `forth' stages of back-and-forth we shall  make the domain of $ j $ to be equal to $ \mm $, and  `back' stages are for making the range of $ j $ to be an initial segment of  $ \M $. For the first step of induction, we will choose $0\mapsto0 $. Then, suppose   $\bar{u}\mapsto\bar{v}$ is built  such that $\mathcal{M}\models (\bar{v}<b \ \wedge \ \mathrm{P}(\bar{u},\bar{v}) \ \wedge \mathrm{Q}(\bar{u},\bar{v}) )$.\\

\textbf{`Forth' stages:} Let $ m\in\mm\setminus\{\bar{u}\} $.  By the definition of $ \M $, without loss of generality, we can assume that $ m\leq t(\bar{u},\p) $ for some $ t\in\F $ and $ \textbf{i}\in I $.   In order to find some image for $ m $, first note that since $ \mathrm{P}(\bar{u},\bar{v}) $ holds in $ \M $,   Theorem 2.3 and Remark 1 imply that:\\

	$ (1):\ \ $ There exists some initial self-embedding $ j_{0} $ of $ \mathcal{M} $ such that $ j_{0}( M)<b $, $ j_{0}(\bar{u})=\bar{v} $, and $  N_{0}\subseteq\mathrm{Fix}(j_{0}) $.\\
	
	Then, we define:\\
	
	$  C:=\left\lbrace
	\langle r,i\rangle\in I: \ \mathcal{M}\models[f_{r}(\bar{u},m,(a)_{i})\downarrow] \ \text{and} \ f_{r}(\bar{u},m,(a)_{i})\notin\mathrm{K}^{1}(\M; N_{0}\cup\{\bar{u}\})
	\right\rbrace  $.\\ 
	
	We claim that $  C\in\mathrm{SSy}_{ I}(\mathcal{M}) $; so there exists  some $\alpha\in M  $ such that $  C= I\cap\alpha_{\mathrm{E}} $. To prove this claim, let:\\
	
	$  R:=\left\lbrace \langle \langle r,i\rangle,k,t\rangle\in I: \ \mathcal{M}\models 
	\left(\begin{array}{c}
	([f_{r}(\bar{u},m,(a)_{i})\downarrow]\wedge [f_{t}(\bar{u},(a)_{k})\downarrow])\rightarrow\\ f_{r}(\bar{u},m,(a)_{i})=f_{t}(\bar{u},(a)_{k})
	\end{array}\right)\right\rbrace $.\\
	
	On one hand, since $  R $ is $ \Pi_{1} $-definable in $ \mathcal{M} $, then $  R\in\mathrm{SSy}_{ I}(\mathcal{M}) $. On the other hand, by Lemma 3.2(2), it holds that:
	\begin{center}
		${ I\setminus C= \overset{ B}{\overbrace{\lbrace  \langle r,i\rangle\in I: \ ( I, R)\models  \exists k,t \ \langle\langle r, i\rangle,k,t\rangle\in R\rbrace}}} $.
	\end{center}
	Since $ I $ is strong in $ \M $, which implies that  $ (I,\ssy_{I}(\M))\models\mathrm{ACA}_{0} $, and because $  B $ is arithmetical in $  R $ and $  R\in\mathrm{SSy}_{ I}(\mathcal{M}) $, we may deduce that $  B\in\mathrm{SSy}_{ I}(\mathcal{M}) $, and consequently $  C\in\mathrm{SSy}_{ I}(\mathcal{M}) $.

	Now, for every $ s\in M $, we define:
	\begin{center}
		$ p_{s}(y):= \lbrace y\leq t(\bar{v},\p)\rbrace\cup p_{s1}(y)\cup p_{s2}(y)$; where: \\
		\vspace{.5cm}
		${p_{s1}(y):=\lbrace\forall i<s([f(\bar{u},m,(a)_{i})\downarrow]\rightarrow [f(\bar{v},y,(a)_{i})\downarrow]^{<b}): f\in\mathcal{F}\rbrace;} $\\ and \\ 
		$p_{s2}(y):=\left\lbrace \forall i<s \left(\begin{array}{c}
		([f_{n}(\bar{v},y,(a)_{i})\downarrow]^{<b} \wedge\langle n,i\rangle\mathrm{E}\alpha)\rightarrow \\
		
		f_{n}(\bar{u},m,(a)_{i})\neq f_{n}(\bar{v},y,(a)_{i})
		\end{array}\right):   \ n\in\mathbb{N} \right\rbrace$.
	\end{center}
	
	We shall show that there is some $ s> I $ such that $ p_{s} $ is finitely satisfiable; then since $ p_{s} $ is $ \Pi_{1} $, bounded and recursive, there exists some $ m' $ which realises  $ p_{s} $ in $ \mathcal{M} $. Therefore, $ m' $ serves as the image of $ m $, and this finishes the `forth' stage.
	
	In order to find such $ s $, we claim that for every $ k\in \mathbb{N} $ it holds that:
	
	$ (\ast_{k}): \\ \ \left(
	\begin{array}{c}
	\text{For every }   f\in\mathcal{F}, \text{ every } z\in N_{0}, \\ \text{ and any nonempty  finite set }  \lbrace f_{n_{0}},...,f_{n_{k}}\rbrace  \text{ of elements of }  \mathcal{F}, \\ \text{ there exists some } s> I \text{ such that } \M\models\Psi(f,f_{n_{0}},...,f_{n_{k}},\bar{u},m,\bar{v},b,a,s,\alpha,z,\p), \\
	\text{ where } \Psi(f,f_{n_{0}},...,f_{n_{k}},\bar{u},m,\bar{v},b,a,s,\alpha,z,\p) \text { is the following } \Pi_{1}\text{-formula: }\\

	{ \exists y\leq t(\bar{v},\p) \left(\begin{array}{c}
		\forall i<s([f(\bar{u},m,(a)_{i},z)\downarrow]\rightarrow[f(\bar{v},y,(a)_{i},z)\downarrow]^{<b}) \ \wedge  \\ \forall i<s \bigwedge_{t\leq k}\left(\begin{array}{c}
		([f_{n_{t}}(\bar{v},y,(a)_{i})\downarrow]^{<b} \wedge	\langle n_{t},i\rangle\mathrm{E}\alpha)\rightarrow \\ 
		f_{n_{t}}(\bar{u},m,(a)_{i})\neq f_{n_{t}}(\bar{v},y,(a)_{i})
		
		\end{array}\right)
		\end{array}\right) }  
	\end{array}\right)$.\\

	This claim completes the proof in the following way:

	Let $ d> I $ be an arbitrary and fixed element of $ \mm $.  Suppose   $i,s\in M$,  and\\ $  \Theta(s,i,\bar{u},m,\bar{v},b,a,\alpha,\beta,\p) $ is the following $ \Delta_{0}$-formula:
	\begin{center}
		${\forall r<i \ \exists y\leq t(\bar{v},\p)  \left(\begin{array}{c}
			\forall w<s (\langle r,w\rangle\mathrm{E}\beta\rightarrow [f_{r}(\bar{v},y,(a)_{w})\downarrow]^{<b}) \ \wedge  
			\\ \forall w<s \forall r'<i\left(\begin{array}{c}
			([f_{r'}(\bar{v},y,(a)_{w})\downarrow]^{<b}\wedge	\langle r',w\rangle\mathrm{E}\alpha)\rightarrow \\ f_{r'}(\bar{u},m,(a)_{w})\neq f_{r'}(\bar{v},y,(a)_{w})
			\end{array}\right)
			\end{array}\right) } $;
	\end{center}
		where $ \beta $ is the code of the following $ \Sigma_{1}$-definable set in $ \mathcal{M} $:\\
	
		$ L:=\lbrace \langle r,w\rangle<d: \ \mathcal{M}\models[f_{r}(\bar{u},m,(a)_{w})\downarrow]\rbrace  $.\\
	
	Now, for every $ i\in M $, we define:
	\begin{center}
		$ g(i):=\mathrm{max}\lbrace x<d: \ \mathcal{M}\models \Theta(x,i,\bar{u},m,\bar{v},b,a,\alpha,\beta,\p) \rbrace $.
	\end{center}
	
	Clearly $ g  $ is $ \Delta_{0}$-definable function in $ \mathcal{M} $, and $ I\subseteq\mathrm{Dom}(g)  $ (we assume $ \mathrm{max}(\emptyset)=0 $). Therefore, since $  I $ is strong, there exists some $ e>  I $ such that for all $ i\in  I $, $ g(i)>  I $ iff $ g(i)>e $. We will show that $ p_{e}(y) $ is a finitely satisfiable type.  First, note that by statement $ (1) $, $ p_{e1}(y) $ is closed under conjunctions. So let $f_{n}, f_{n_{0}},...,f_{n_{k}} $ be some finite number of elements of $ \mathcal{F} $, and let $ n^{*}=\mathrm{max}\lbrace n,n_{0},...,n_{k}\rbrace $.  Then, use $(\ast_{n^{*}}) $, $ (n^{*}+2)$-many times; i.e for every $ t=0,...,n^{*}+1 $  consider $ f_{t} $ instead of $ f $ in the assertion of $ (\ast_{n^{*}})  $, $ 0\in M_{0} $ instead of $ z $, and $ f_{1},...,f_{n^{*}} $.  So by statement $ (\ast_{n^{*}})  $, for every $ t=0,...,n^{*}+1 $  there exists some $ s_{t}> I $ such that  $ \M\models\Psi(f_{t},f_{0},...,f_{n^{*}+1},\bar{u},m,\bar{v},b,a,s_{t},\alpha,0,\p) $. Then, let $ s^{*}:=\mathrm{min}\lbrace s_{t}: \ t<n^{*}+1\rbrace $. Therefore,  $ {\mathcal{M}\models \Theta(s^{*},n^{*},\bar{u},m,\bar{v},b,a,\alpha,\beta,\p)} $. \\ It is easy to see that if $ d\leq s^{*}$ then $ g(n^{*})=d-1 $, and if $ s^{*}<d $ then $ s^{*}\leq g(n^{*}) $; so in both cases $ g(n^{*})> I $ and consequently $ g(n^{*})>e $. So $\mathcal{M}\models\Theta(e,n^{*},\bar{u},m,\bar{v},b,a,\alpha,\beta,\p)$; this proves that $ p_{e} $ is finitely satisfiable.\\
	
	\textit{Proof of the claim $ (\ast_{k}) $ for every $ k\in\mathbb{N} $:} Suppose the claim is not true; i.e there is some $ k\in\mathbb{N} $ for which there exists some nonempty finite set $ \lbrace f,f_{n_{0}},...,f_{n_{k}}\rbrace $ of elements of $ \mathcal{F} $, and some $ \textbf{z}\in N_{0} $ such that for all  $ s> I $ it holds that: 
	$$ \mathcal{M}\models\neg  \Psi(f,f_{n_{0}},...,f_{n_{k}},\bar{u},m,\bar{v},b,a,s,\alpha,\textbf{z},\p). $$

	Therefore, by $ \Sigma_{1}$-Underspill in $ \mathcal{M} $, there exists some $ s\in I $ such that:\\
	
	$$ \mathcal{M}\models \neg\Psi(f,f_{n_{0}},...,f_{n_{k}},\bar{u},m,\bar{v},b,a,s,\alpha,\textbf{z},\p). $$
	
	$ (2):\ \ $ Let $ k_{0}\in\mathbb{N} $  be the least natural number, for which
	there exists a set  $\lbrace f, f_{n_{0}},...,f_{n_{k_{0}}}\rbrace $ of  elements of $ \mathcal{F} $, some $ \textbf{z}_{0}\in N_{0} $, and some $ s_{0}\in I $ such that:  $$ {\mathcal{M}\models \neg\Psi(f,f_{n_{0}},...,f_{n_{k_{0}}},\bar{u},m,\bar{v},b,a,s_{0},\alpha,\textbf{z}_{0},\p)}. $$
	
		Put:
	\begin{center}
		$  X:=\lbrace x\in M: \ \mathcal{M}\models\exists i<s_{0}(x=(a)_{i}\wedge [f(\bar{u},m,(a)_{i},\textbf{z}_{0})\downarrow] ) \rbrace $; \\and\\
		$ X':=\lbrace \langle n,x\rangle\in M: \ \mathcal{M}\models\exists i<s_{0}(x=(a)_{i}\wedge \bigvee_{t=0}^{k_{0}} n=n_{t} \wedge \langle n,i\rangle\mathrm{E}\alpha) \rbrace. $
	\end{center}
	By Lemma 3.3, there exist $ (a)_{\xi}\in N_{0} $ and $ (a)_{\zeta}\in N_{0} $ which code $  X $ and $ X' $ respectively. So we can restate statement (2) in the following form:\\
	
	$ (3):\ \ $ Let $ k_{0}\in\mathbb{N} $  be the least natural number, for which
	there exists a set  $\lbrace f, f_{n_{0}},...,f_{n_{k_{0}}}\rbrace $ of  elements of $ \mathcal{F} $, some $ \textbf{z}_{0},(a)_{\zeta},(a)_{\xi}\in N_{0} $ such that:\\

	$\M\models\forall y\leq t(\bar{v},\p) \left( \begin{array}{c}
	\forall \epsilon<(a)_{\xi}( \epsilon\mathrm{E}(a)_{\xi}\rightarrow[f(\bar{v},y,\epsilon,\textbf{z}_{0})\downarrow]^{<b} \rightarrow  
	\\\exists  \varepsilon<\mathrm{E}(a)_{\zeta} \bigvee_{t=0}^{k_{0}}
	\left(\begin{array}{c}
	\langle n_{t},\varepsilon\rangle\mathrm{E}(a)_{\zeta} \ \wedge	[f_{n_{t}}(\bar{v},y,\varepsilon)\downarrow]^{<b} \wedge \\
	
	f_{n_{t}}(\bar{u},m,\varepsilon)= f_{n_{t}}(\bar{v},y,\varepsilon)

	\end{array}\right)
	\end{array}\right)  $.

	Now, by considering the sequence number of ${ \langle f_{n_{t}}(\bar{u},m,\varepsilon): \ \langle n_{t},\varepsilon\rangle\E(a)_{\zeta} \rangle} $ in $ \mathcal{M} $, we may quntify out $ f_{n_{t}}(\bar{u},m,\varepsilon) $s  from the formula in statement $ (3) $, and deduce that:\\
	
	$ (4): \ \ $ $ \mathcal{M}\models \exists x 
	\forall y\leq t(\bar{v},\p) \ \theta(y,b,\bar{v},x,(a)_{\xi},(a)_{\zeta},\textbf{z}_{0}) $, where $\theta(y,b,\bar{v},x,(a)_{\xi},(a)_{\zeta},\textbf{z}_{0})  $ is the following $ \Delta_{0} $-formula: 
	
	$$	\left( \begin{array}{c}
			\forall \epsilon<(a)_{\xi}( \epsilon\mathrm{E}(a)_{\xi}\rightarrow[f(\bar{v},y,\epsilon,\textbf{z}_{0})\downarrow]^{<b}) \rightarrow  
			\\ \exists \langle  n_{t},\varepsilon\rangle\mathrm{E}(a)_{\zeta}
			\left(\begin{array}{c} [f_{n_{t}}(\bar{v},y,\varepsilon)\downarrow]^{<b}\wedge (x)_{\langle n_{t},\varepsilon\rangle}= f_{n_{t}}(\bar{v},y,\varepsilon)
			\end{array}\right)
			\end{array}\right).$$

	Then, we will define  $ \Sigma_{1}$-definable  partial functions $ b(\diamondsuit,y,(a)_{\xi},(a)_{\zeta},\textbf{z}_{0}) $ and $s(\diamondsuit,(a)_{\xi},(a)_{\zeta},\textbf{z}_{0},\p) $, as follows (we omit the parameters $(a)_{\xi}  $, $ (a)_{\zeta} $, $ \p $, and $ \textbf{z}_{0} $ in the presentations of these functions for the sake of simplicity):
	\begin{itemize}

		\item	$b(\diamondsuit,y):=\min\left\lbrace w: \left(\begin{array}{c}
		\forall  \epsilon\mathrm{E}(a)_{\xi} \ ([f(\diamondsuit,y,\epsilon,\textbf{z}_{0})\downarrow]^{<w})\wedge \\ 
	\forall \langle  n_{t},\varepsilon\rangle\E(a)_{\zeta} \ ([f_{n_{t}}(\diamondsuit,y,\varepsilon)\downarrow]^{<w})
	
		\end{array}\right)  \right\rbrace. $
	\item	$ s(\diamondsuit):=x  $ iff $ \exists z  \left(\begin{array}{c}(z)_{0}=x \ \wedge\\ z=\mu_{w} \ \forall y\leq t(\diamondsuit,\p) \ \left(\begin{array}{c}[b(\diamondsuit,y)\downarrow]^{<(w)_{1}} \ \rightarrow\\ \theta(y,b(\diamondsuit,y),\diamondsuit,(w)_{0},(a)_{\xi},(a)_{\zeta},\textbf{z}_{0})\end{array}\right)\end{array}\right) ; $\\
		
		and $ s_{t}(\diamondsuit,\varepsilon):=(s(\diamondsuit))_{\langle  n_{t},\varepsilon\rangle} $, for every $ \langle  n_{t},\varepsilon\rangle\E(a)_{\zeta} $.

	\end{itemize}
	
 From the definition  of $ s_{t}(\bar{v},\varepsilon) $s and statement (4) we may infer that:

	$ (5): \ \ $ $ \mathcal{M}\models  
	\forall y\leq t(\bar{v},\p) \left( \begin{array}{c}
([b(\bar{v},y)\downarrow]^{<b} \ \wedge \	\forall \epsilon<(a)_{\xi}( \epsilon\mathrm{E}(a)_{\xi}\rightarrow[f(\bar{v},y,\epsilon,\textbf{z}_{0})\downarrow]^{<b(\bar{v},y)})) \rightarrow  
	\\ \exists \langle  n_{t},\varepsilon\rangle\mathrm{E}(a)_{\zeta}
	\left(\begin{array}{c} [f_{n_{t}}(\bar{v},y,\varepsilon)\downarrow]^{<b(\bar{v},y)} 
	\wedge [s_{t}(\bar{v},\varepsilon)\downarrow]^{<b(\bar{v},y)} \ \wedge\\
	s_{t}(\bar{v},\varepsilon)= f_{n_{t}}(\bar{v},y,\varepsilon)
	\end{array}\right)
	\end{array}\right)$.\\
	
	It is not difficult to express the formula in the statement (5)  in the form of $ {\forall z<b \ \delta(\bar{v},(a)_{\xi},(a)_{\zeta},\textbf{z}_{0})} $ for some $ \Delta_{0}$-formula $ \delta $. Therefore, by the property  $ \mathrm{P}(\bar{u},\bar{v}) $, the definition of function $ s $, and statement (5)  we deduce that:\\

	$ (6): \ \ $$ \mathcal{M}\models  
 \forall y\leq t(\bar{u},\p) \left( \begin{array}{c}
([b(\bar{u},y)\downarrow] \ \wedge \	\forall \epsilon<(a)_{\xi}( \epsilon\mathrm{E}(a)_{\xi}\rightarrow[f(\bar{u},y,\epsilon,\textbf{z}_{0})\downarrow]^{<b(\bar{u},y)}) )\rightarrow  
	\\ \exists \langle  n_{t},\varepsilon\rangle\mathrm{E}(a)_{\zeta}
	\left(\begin{array}{c} [f_{n_{t}}(\bar{u},y,\varepsilon)\downarrow]^{<b(\bar{u},y)} \wedge
	
	[s_{t}(\bar{u},\varepsilon)\downarrow]^{<b(\bar{u},y)} ) \ \wedge\\ s_{t}(\bar{u},\varepsilon)= f_{n_{t}}(\bar{u},y,\varepsilon)
	\end{array}\right)
	\end{array}\right)$.\\
	
	Now, we will simultaneously define  two more $ \Sigma_{1}$-definable functions in $ \M $: \\
	
	$ \langle o(\diamondsuit,y),h(\diamondsuit,y)\rangle:=\min\left\lbrace\langle n_{t},\varepsilon\rangle\E(a)_{\zeta}: 
	\left(\begin{array}{c}
	[b(\diamondsuit,y)\downarrow]\wedge\\
	
 [f_{n_{t}}(\diamondsuit,y,\varepsilon)\downarrow]^{<b(\diamondsuit,y)} \ \wedge
	
	[s_{t}(\diamondsuit,\varepsilon)\downarrow]^{<b(\diamondsuit,y)} \wedge\\ s_{t}(\diamondsuit,\varepsilon)= f_{n_{t}}(\diamondsuit,y,\varepsilon)\end{array}\right) 
	\right\rbrace $.\\
	
	(Note that, similar to the way we defined function $ s $, we can express the above definition by a $ \Sigma_{1}$-formula.) Then, by statement (5)  it holds that:\\
	
	$ (7):\ \ $ $ \M\models\forall y\leq t(\bar{v},\p)\left(\begin{array}{c} ([b(\bar{v},y)\downarrow]^{<b} \ \wedge \ \forall \epsilon<(a)_{\xi}( \epsilon\mathrm{E}(a)_{\xi}\rightarrow[f(\bar{v},y,\epsilon,\textbf{z}_{0})\downarrow]^{<b})) \rightarrow \\ 
	
	[\langle o(\bar{v},y),h(\bar{v},y)\rangle\downarrow]\end{array}\right)  $.\\
	
	Similarly, from statement (6) we may deduce that:\\
	
	$ (8):\ \ $ $ \M\models\forall y\leq t(\bar{u},\p)\left(\begin{array}{c} ([b(\bar{u},y)\downarrow] \ \wedge \ \forall \epsilon<(a)_{\xi}( \epsilon\mathrm{E}(a)_{\xi}\rightarrow[f(\bar{u},y,\epsilon,\textbf{z}_{0})\downarrow]^{<b(\bar{u},y)})) \rightarrow \\ 
	
	[\langle o(\bar{u},y),h(\bar{u},y)\rangle\downarrow]\end{array}\right)  $.\\

	Finally, we obtain a contradiction by dividing $ k_{0} $ into two cases in the following way:
	\begin{itemize}
		\item \underline{If  $ k_{0}=1 $}, we inductively define the following $ \Sigma_{1}$-function in $ \M $: \\
		
		$ w(\diamondsuit,0):=\min\left\lbrace y\leq t(\diamondsuit,\p): \  \left(\begin{array}{c}[b(\diamondsuit,y)\downarrow]\ \wedge \ 	[h(\diamondsuit,y)\downarrow]^{<(a)_{\zeta}} \wedge \\ \forall \epsilon<(a)_{\xi}( \epsilon\mathrm{E}(a)_{\xi}\rightarrow[f(\diamondsuit,y,\epsilon,\textbf{z}_{0})\downarrow]^{<b(\diamondsuit,y)} )
		\end{array}\right) \right\rbrace $,\\ and \\
		
		$ w(\diamondsuit,i+1):=$  $\min\left\lbrace y\leq t(\diamondsuit,\p): \varphi(\diamondsuit,i,y,(a)_{\zeta},(a)_{\xi},\textbf{z}_{0})\right\rbrace $, where $ \varphi(\diamondsuit,i,y,(a)_{\zeta},(a)_{\xi},\textbf{z}_{0}) $ is the following formula:\\
		
		$ \left(\begin{array}{c}  [b(\diamondsuit,y)\downarrow]\ \wedge [h(\diamondsuit,y)\downarrow]^{<(a)_{\zeta}} \wedge \\ \forall \epsilon<(a)_{\xi}( \epsilon\mathrm{E}(a)_{\xi}\rightarrow[f(\diamondsuit,y,\epsilon,\textbf{z}_{0})\downarrow]^{<b(\diamondsuit,y)}) \ \wedge \\
		\forall x\leq i  \left(\begin{array}{c}
		\left(\begin{array}{c}
		
		[h(\diamondsuit,w(\diamondsuit,x))\downarrow]^{<(a)_{\zeta}}\wedge\\	
		
		[f_{n_{0}}(\diamondsuit,y,h(\diamondsuit,w(\diamondsuit,x)))\downarrow]^{<b(\diamondsuit,y)}\wedge\\
		
		[f_{n_{0}}(\diamondsuit,w(\diamondsuit,x),h(\diamondsuit,w(\diamondsuit,x)))\downarrow]^{<b(\diamondsuit,w(\diamondsuit,x))}\end{array}\right) \rightarrow \\
		
		f_{n_{0}}(\diamondsuit,y,h(\diamondsuit,w(\diamondsuit,x)))\neq f_{n_{0}}(\diamondsuit,w(\diamondsuit,x),h(\diamondsuit,w(\diamondsuit,x))) 
		\end{array}\right)\end{array}\right).$\\ 
		
			First, we will show that  $ \M\models[w(\bar{u},i)\downarrow] $ for all $ i\in I $.  Otherwise, there exists the least $ 0<i_{0}\in I $ such that:\\
			
				$ (9): \ \ $$ \M\models \forall y\leq t(\bar{u},\p) \ \neg \varphi(\bar{u},i_{0},y,(a)_{\zeta},(a)_{\xi},\textbf{z}_{0})$.\\
			
		Note that  by the definition of $ (a)_{\xi} $ and $ (a)_{\zeta} $ it holds that:\\
		
			$ (10):\ \ $ $ \M\models ([b(\bar{u},m)\downarrow] \ \wedge \ \forall \epsilon<(a)_{\xi}( \epsilon\mathrm{E}(a)_{\xi}\rightarrow[f(\bar{u},m,\epsilon,\textbf{z}_{0})\downarrow]^{<b(\bar{u},m)})    $.\\

		 So by statements (8), (9) and (10), there exists some $ i_{1}< i_{0} $ such that:\\
			
			$ (11): \ \ $$ \M\models  
			\left(\begin{array}{c}
			[h(\bar{u},w(\bar{u},i_{1}))\downarrow] \wedge [f_{n_{0}}(\bar{u},m,h(\bar{u},w(\bar{u},i_{1})))\downarrow] \wedge \\
			
			[f_{n_{0}}(\bar{u},w(\bar{u},i_{1}),h(\bar{u},w(\bar{u},i_{1})))\downarrow] \wedge \\
			f_{n_{0}}(\bar{u},m,h(\bar{u},w(\bar{u},
i_{1})))= f_{n_{0}}(\bar{u},w(\bar{u},i_{1}),h(\bar{u},w(\bar{u},i_{1}))) 
			\end{array}\right)  $.\\
			
			Clearly, $f_{n_{0}}(\bar{u},w(\bar{u},i_{1}),h(\bar{u},w(\bar{u},i_{1})))\in\mathrm{K}^{1}(\M;N_{0}\cup\{\bar{u}\})  $. So by statement (11), $f_{n_{0}}(\bar{u},m,h(\bar{u},w(\bar{u},i_{1}))) \in\mathrm{K}^{1}(\M;\mm_{0}\cup\{\bar{u}\}) $. So $ \M\models\neg\langle n_{0},h(\bar{u},w(\bar{u},i_{1}))\rangle\E(a)_{\zeta} $ (by the definition of $ (a)_{\zeta} $), which is in contradiction with the definition of the function $ h $.
			
				As a result,  by the definition of $ w(\bar{u},i) $ and statement (8), the function $ i\mapsto h(\bar{u},w(\bar{u},i)) $ from $\{i: \ i\leq s_{0}+1\} $ into $ ((a)_{\zeta})_{\E} $ is well-defined and coded in $ \M $. So, since the cardinality of $ (a)_{\zeta} $ is less than $ s_{0}+1 $, by $ \Sigma_{1}$-Pigeonhole Principle in $ \M $, there exists some distinct $ i_{0}<i_{1}\leq s_{0}+1 $ such that:\\
			
			$ (12): \ \ $ $ \M\models h(\bar{u},w(\bar{u},i_{0}))=h(\bar{u},w(\bar{u},i_{1}))  $.\\
			
			Therefore, by statement (12) and the definition of $ h $ we conclude that: \\
			
			$ (13): \ \ $ $ \M\models\left(\begin{array}{c} [s_{0}(\bar{u},h(\bar{u},w(\bar{u},i_{0})))\downarrow] \wedge [s_{0}(\bar{u},h(\bar{u},w(\bar{u},i_{1})))\downarrow]\wedge\\ s_{0}(\bar{u},h(\bar{u},w(\bar{u},i_{0})))=s_{0}(\bar{u},h(\bar{u},w(\bar{u},i_{1})))	\end{array}\right) $.\\

			Moreover, by the definition of $ h(\bar{u},w(\bar{u},i)) $, for   $ i=i_{0},i_{1} $ it holds that:\\
			
				$ (14): \ \ $  $ \M\models\left(\begin{array}{c}
			[f_{n_{0}}(\bar{u},w(\bar{u},i),h(\bar{u},w(\bar{u},i)))\downarrow] \ \wedge[s_{0}(\bar{u},h(\bar{u},w(\bar{u})))\downarrow] \wedge\\ f_{n_{0}}(\bar{u},w(\bar{u},i),h(\bar{u},w(\bar{u},i)))=s_{0}(\bar{u},h(\bar{u},w(\bar{u}))) 
			\end{array}\right)$.\\
			
			So statements (12), (13) and (14) imply that:\\
			
			$ (15): \ \ $ $ \M\models f_{n_{0}}(\bar{u},w(\bar{u},i_{1}),h(\bar{u},w(\bar{u},i_{0})))=f_{n_{0}}(\bar{u},w(\bar{u},i_{0}),h(\bar{u},w(\bar{u},i_{0}))) $. \\
			
			But statement (15) is in contradiction with the definition of $ w(\bar{u},i_{1}) $.

		\item \underline{If   $ k_{0}>1 $}, by using Lemma 3.3, let $ (a)_{\rho}\in N_{0} $ be the code of the following subset of  $ N_{0} $:\\
		
		$ \A:=\left\lbrace\langle o(\bar{v},y),h(\bar{v},y)\rangle: \  \M\models \left(\begin{array}{c}  y\leq t(\bar{v},\p)  \wedge	[\langle o(\bar{v},y),h(\bar{v},y)\rangle\downarrow]^{<b} \wedge\\ \forall \epsilon<(a)_{\xi}( \epsilon\mathrm{E}(a)_{\xi}\rightarrow[f(\bar{v},y,\epsilon,\textbf{z}_{0})\downarrow]^{<b} \wedge\\
		\exists \varepsilon<(a)_{\zeta} \left(\begin{array}{c} \langle n_{0},\varepsilon\rangle\E(a)_{\zeta} \wedge [f_{n_{0}}(\bar{v},y,\varepsilon)\downarrow]^{<b} \wedge \\
		f_{n_{0}}(\bar{v},y,\varepsilon)=f_{n_{0}}(\bar{u},m,\varepsilon)
		\end{array}\right) \end{array}\right) \right\rbrace $.\\
		
		So, by statements (3), (7), and the definition of  $ (a)_{\rho} $, we conclude that:\\
		
		$ (16): \ \ $$ \mathcal{M}\models \forall y\leq t(\bar{v},\p) \left( \begin{array}{c}
		\forall \epsilon<(a)_{\xi}\left( \begin{array}{c} \epsilon\mathrm{E}(a)_{\xi}\rightarrow[f(\bar{v},y,\epsilon,\textbf{z}_{0})\downarrow]^{<b} \wedge  \\ 

		[\langle o(\bar{v},y),h(\bar{v},y)\rangle\downarrow]^{<b} \wedge\\
		\neg \langle o(\bar{v},y),h(\bar{v},y)\rangle\E(a)_{\rho} \end{array}\right) \rightarrow  
		\\ \exists\varepsilon<\mathrm{E}(a)_{\zeta}\bigvee_{t=1}^{k_{0}}
		
		\left(\begin{array}{c}
		\langle n_{t},\varepsilon\rangle\mathrm{E}(a)_{\zeta}\wedge
		
		[f_{n_{t}}(\bar{v},y,\varepsilon)\downarrow]^{<b}\wedge\\ f_{n_{t}}(\bar{u},m,\varepsilon)= f_{n_{t}}(\bar{v},y,\varepsilon)

		\end{array}\right)
		\end{array}\right)  $.\\
		
		Let $ f'\in\F $ such that:\\
		\begin{center}
			$ f'(\diamondsuit,y,\epsilon,\langle\textbf{z}_{0},(a)_{\rho},(a)_{\zeta},(a)_{\xi}\rangle)=x $ \\iff \\	$ 
			x=f(\diamondsuit,y,\epsilon,\textbf{z}_{0})\ \wedge [\langle o(\diamondsuit,y),h(\diamondsuit,y)\rangle\downarrow] \wedge  \neg \langle o(\diamondsuit,y),h(\diamondsuit,y)\rangle\E(a)_{\rho} 
			$.\\
		\end{center}	
		So by considering $ f' $ instead of $ f $ in  statement (3), statement (16)  leads to contradiction with the minimality of $ k_{0} $.
			\end{itemize}

\textbf{	`Back' stages:} Let $ m'\in M\setminus\{\bar{v}\} $ such that  $ m'<v_{0}:=\max\{\bar{v}\} $, and $ u_{0}:=\max\{\bar{u}\} $. In order to find some element of $ \mm $ whose image is $ m' $, we modify the proof of the `forth' stage  in the following way:
	\begin{itemize}

		\item Let $ \alpha' $ be the code of the following set in $ \mathcal{M} $:\\
		
		$ {C':=\lbrace\langle r,i\rangle\in I:  \ \mathcal{M}\models[f_{r}(\bar{v},m',(a)_{i})\downarrow]^{<b} \ \text{and} \ f_{r}(\bar{v},m',(a)_{i})\notin\mathrm{K}^{1}(\M; N_{0}\cup\{\bar{v}\})\rbrace} $.
		\item Replace $ p_{s}(y) $ by:
		\begin{center}
			$ q_{s}(x):= \lbrace x<u_{0}\rbrace\cup q_{s1}(x)\cup q_{s2}(x) $; where: \\ 
			\vspace{.5cm}
			${q_{s1}(x):= \lbrace\forall i<s(\neg[f(\bar{v},m',(a)_{i})\downarrow]^{<b}\rightarrow\neg[f(\bar{u},x,(a)_{i})\downarrow] ): f\in\mathcal{F}\rbrace;} $ and\\ 
			${q_{s2}(x):=\left\lbrace\begin{array}{c}
				\forall i<s\left(\begin{array}{c}  
				\left(\begin{array}{c}\langle n,i\rangle\mathrm{E}\alpha'\wedge\\ 
				
				[f_{n}(\bar{u},x,(a)_{i})\downarrow]\end{array}\right)\rightarrow \\ f_{n}(\bar{u},x,(a)_{i})\neq f_{n}(\bar{v},m',(a)_{i})
				\end{array}\right): 
				\ n\in\mathbb{N}	
				\end{array}\right\rbrace} $.
		\end{center}
		\item Let: \\
		$ (\ast'_{k}): \ \ \left(
		\begin{array}{c}
		\text{For every }   f\in\mathcal{F}, \text{ every } z\in N_{0}, \\ \text{ and any nonempty finite set }  \lbrace f_{n_{0}},...,f_{n_{k}}\rbrace  \text{ of elements of  } \mathcal{F}, \\ \text{ there exists some } s> I \text{ such that:}\\
		\vspace{.3cm}
		{ \mathcal{M}\models \exists x<u_{0} \left(\begin{array}{c}
			\forall i<s(\neg[f(\bar{v},m',(a)_{i},z)\downarrow]^{<b}\rightarrow\neg[f(\bar{u},x,(a)_{i},z)\downarrow]) \ \wedge  \\\forall i<s \bigwedge_{t\leq k}\left(\begin{array}{c}
			(\langle n_{t},i\rangle\mathrm{E}\alpha'\wedge[f_{n_{t}}(\bar{u},x,(a)_{i})\downarrow]) \rightarrow \\ f_{n_{t}}(\bar{u},x,(a)_{i})\neq f_{n_{t}}(\bar{v},m',(a)_{i})
			\end{array}\right)
			\end{array}\right) }  
		\end{array}\right)$.
		\item Replace $ \Theta(s,i,\bar{u},m,\bar{v},b,a,\alpha,\beta) $ with $  \Theta'(s,i,\bar{u},\bar{v},m',b,a,\alpha',\beta') $:
		\begin{center}
			$\forall r<i \ \exists x<u_{0}\left(\begin{array}{c}
			\forall w<s( \langle x,r,w\rangle\mathrm{E}\beta'\rightarrow[f_{r}(\bar{v},m',(a)_{w})\downarrow]^{<b})  \wedge \\
			\forall w<s \forall r'<i
			\left(\begin{array}{c}
			(\langle r',w\rangle\mathrm{E}\alpha'\wedge\langle x,r',w\rangle\mathrm{E}\beta')\rightarrow\\ f_{r'}(\bar{u},m,(a)_{w})\neq f_{r'}(\bar{v},y,(a)_{w}))
			\end{array}\right)
			\end{array}\right)  $;
	\end{center}	
			where $ \beta' $ is the code of the following $ \Sigma_{1}$-definable set in $ \mathcal{M} $:\\
			
			$ L':=\lbrace\langle x,r,w\rangle: \ \mathcal{M}\models (x<u_{0} \wedge w<d \wedge r<d \wedge [f_{r}(\bar{u},x,(a)_{w})\downarrow]) \rbrace $.
		
		\item Between statements (3) and (4) we need to use $ \Sigma_{1}$-Collection to deduce:\\
		
		$ (3'): \ \ $ $ \mathcal{M}\models \exists w \forall x<u_{0} \left(\begin{array}{c}
		\forall \epsilon<(a)_{\lambda} \ ([f(\bar{u},x,\epsilon,\textbf{z}_{0})\downarrow]^{<w}\rightarrow \epsilon\mathrm{E}(a)_{\lambda})\rightarrow \\ \exists \varepsilon<(a)_{\eta} \bigvee_{t\leq k_{1}}\left(\begin{array}{c}
		\langle n_{t},\varepsilon\rangle\mathrm{E}(a)_{\eta}\wedge [f_{n_{t}}(\bar{u},x,\varepsilon)\downarrow]^{<w}\wedge\\ f_{n_{t}}(\bar{u},x,\varepsilon)= f_{n_{t}}(\bar{v},m',\varepsilon)
		\end{array}\right)
		\end{array}\right)  $;\\
		
		in which $(a)_{\lambda}\in N_{0} $ and $ (a)_{\eta}\in N_{0} $	code the following $ Y $ and $ Y' $ respectively:\\
		
		${ Y:=\lbrace x\in M: \ \mathcal{M}\models\exists i<s_{0}(x=(a)_{i}\wedge[f(\bar{v},m',(a)_{i},\textbf{z}_{0})\downarrow]^{<b} ) \rbrace} $,  and \\
		$ {Y':=\lbrace \langle n,x\rangle\in M: \ \mathcal{M}\models\exists i<s_{0}(x=(a)_{i}\wedge \bigvee_{t=0}^{k_{0}}n=n_{t} \wedge \langle n,i\rangle\mathrm{E}\alpha') \rbrace} $. 
	\end{itemize}
	The rest of the argument goes smoothly by modifying  the `forth' stage  according to the above changes, and this completes the proof.
\end{proof}

\begin{cor}
	Assume $ \mathcal{M}\models \mathrm{I}\Sigma_{1} $ is countable and nonstandard, $  I $ is a proper cut of $ \mathcal{M} $, and $ \mathcal{M}_{0} $ is an $  I$-small $ \Sigma_{1} $-elementary submodel of  $ \mathcal{M} $. Then the following are equivalent:
	\begin{itemize}
		\item[1)] $  I $ is strong in $ \mathcal{M} $.
		\item[2)] There exists some proper initial self-embedding $ j $ of $ \mathcal{M} $ such that $ { M_{0}=\mathrm{Fix}(j)} $.
	\end{itemize}
\end{cor}
\begin{proof} Suppose ${  M_{0}=\lbrace (a)_{i}: \ i\in I\rbrace} $, for some $ a\in M $ such that $ (a)_{i}\neq(a)_{j} $ for all distinct $ i, j\in I $.\\
	
	\textbf{$ (1)\Rightarrow(2) $:} If $ \mm_{0}=I $, then by Theorem 2.4(2), we are done. So suppose $ I\subsetneq\mm_{0} $. First, by using Theorem 3.4 let $ h  $ be  some proper initial self-embedding  of $ \mathrm{H}^{1}(\M;\mm_{0}) $ such that $ \fix(h)=\mm_{0} $. Moreover, fix some $b\in \mathrm{H}^{1}(\M;\mm_{0})\setminus\mm_{0}  $ such that $ h( \mathrm{H}^{1}(\M;\mm_{0}))<b $. Now, by using strong  $ \Sigma_{1}$-Collection in $ \mathrm{H}^{1}(\M;\mm_{0}) $, and since $ \mathrm{H}^{1}(\M;\mm_{0})\prec_{\Sigma_{1}}\M $,  we can find some $ d\in\mathrm{H}^{1}(\M;\mm_{0}) $ such that:
	\begin{center}
		$ \mathcal{M}\models [f((a)_{i},b)\downarrow]\rightarrow[f((a)_{i},b)\downarrow]^{<d} $, for all $ f\in\mathcal{F} $ and all $ i\in I $.
	\end{center}

Therefore, by Theorems 2.1 and 2.3 and Remark 1, there exists some proper initial embedding $ k:\M\hookrightarrow\mathrm{H}^{1}(\M;\mm_{0}) $ such that $ \mm_{0}\subseteq\fix(k) $, $ k(\mm)<d $ and $ b\in k(M) $ (note that since $ \mathrm{H}^{1}(\M;\mm_{0}) $ is an initial segment of $ \M $, then $ \ssy_{I}(\M)=\ssy_{I}(\mathrm{H}^{1}(\M;\mm_{0})) $). Finally, we put $ j:=k^{-1}hk $. It is easy to check that $ j $ is a well-defined proper initial self-embedding of $ \M $ such that $ \fix(j)=\mm_{0} $.

	\textbf{$ (2)\Rightarrow(1) $:} We combine the methods used in the proof of Theorem 5.1 and 6.1 of  \cite{our}. Suppose $ I $ is  not strong; i.e. there exists some coded function $ f $ in $ \mathcal{M} $ such that $  I\subseteq\mathrm{Dom}(f) $, and the set $ {D:=\lbrace f(i): \ i\in I \wedge  I<f(i)\rbrace} $ is downward cofinal in $  M\setminus I $.\\
	Let $ b\in M\setminus M_{0} $ and $ g:=j(f) $. For every $ k\in M $, we put:
	\begin{center}
		$  A_{k}:=\lbrace\langle r,y\rangle<k: \ \mathcal{M}\models\mathrm{Sat}_{\Delta_{0}}(\delta_{r}((a)_{y},b))\rbrace $.
	\end{center}
	Since $  A_{k} $ is bounded and $ \Delta_{1}$-definable, it is coded by some $ s_{k} $ in $ \mathcal{M} $. Moreover, the function $ k\mapsto s_{k} $ is $ \Delta_{1}$-definable in $   \M $. Now, we define:
	\begin{center}
		
		$ h(k):=\mu_{x} \ (\forall\langle r,y\rangle<k \ (\langle r,y\rangle\mathrm{E}s_{k}\rightarrow\mathrm{Sat}_{\Delta_{0}}(\delta_{r}((a)_{y},x)))) $.
	\end{center}
	So note that:
	\begin{itemize}
		\item[(I)] For every $ k> I $, we have $ \mathrm{Th}_{\Delta_{0}}(\mathcal{M};b,\{(a)_{i}\}_{i\in I})\subseteq \mathrm{Th}_{\Delta_{0}}(\mathcal{M};h(k),\{(a)_{i}\}_{i\in I})$.
		\item[(II)] For every $ i\in I $, $ h(i)$ is well-defined and inside $ M_{0}=\mathrm{Fix}(j) $; the reason behind this statement is that  for every $ i\in I $  we consider the following set:
		\begin{center}
			$  B_{i}:=\lbrace\langle r,\epsilon\rangle: \ \mathcal{M}\models \exists y<i ((a)_{y}=\epsilon\wedge \langle r,y\rangle\mathrm{E}s_{i})\rbrace $.
		\end{center}
	Then, 	  by Lemma 3.3, $ B_{i} $ is coded by some $ \alpha_{i}\in M_{0}=\mathrm{Fix}(j) $. So it holds that:
		\begin{center}
			$ \mathcal{M}\models h(i)=\mu_{x}(\forall \langle r,\epsilon\rangle\mathrm{E}\alpha_{i} \ (\mathrm{Sat}_{\Delta_{0}}(\delta_{r}(\epsilon,x))) $.
		\end{center}
		As a result, since $ \mathcal{M}_{0}\prec_{\Sigma_{1}}\mathcal{M} $, statement (II) holds.
	\end{itemize}
	Now, let $ h':=j(h) $. So for all $ i\in I $, and all $ u<i $ such that $ f(u)<i $, statement (II) implies that: $$ h'(g(u))=j(h)(j(f)(u))=j(h)(j(f)(j(u)))=j(h(f(u)))=h(f(u)). $$ Therefore, for all $ i\in I $,  $ \mathcal{M}\models\varphi(i,f,g,h,h') $, where $ \varphi(i,f,g,h,h') $ is the following $ \Delta_{1}$-formula:
	\begin{center}
		$ \forall u<i \ (f(u)<i\rightarrow h(f(u))=h'(g(u)))$.
	\end{center}
	So by $ \Sigma_{1} $-Overspill in $ \mathcal{M} $, there exists some $ s> I $ such that:
	\begin{center}
		$ (\spadesuit): \ \ $ $ \forall u<s \ (f(u)<s\rightarrow h(f(u))=h'(g(u)))$.
	\end{center}
	Since $ D $ is downward cofinal in $  M\setminus I $, there is some $ i_{0}\in I $ such that $  I<f(i_{0})<s $. Let $ c:=h(f(i_{0})) $. On one hand, by (I), $ \mathrm{Th}_{\Delta_{0}}(\mathcal{M};b,\{(a)_{i}\}_{i\in I})\subseteq \mathrm{Th}_{\Delta_{0}}(\mathcal{M};c,\{(a)_{i}\}_{i\in I})$. As a result, because $ b\notin M_{0} $, we have  $ c\notin M_{0}=\mathrm{Fix}(j) $. On the other hand $ (\spadesuit) $ implies that:
	\begin{center}
		$ j(c)=j(h(f(i_{0})))=j(h)(j(f)((j(i_{0})))=h'(g(i_{0}))=h(f(i_{0}))=c $.
	\end{center}
As a result, $ I $ has to be strong in $ \M $.	

 \end{proof}


\section{Strongness of the standard cut and fixed points}
In this section,  we will show some properties of $ \fix(j) $, when $ \mathbb{N} $ is not strong in $ \M $. Then we will  conclude some criteria for stongness of $ \mathbb{N} $ in a countable nonstandard  model of $ \I\Sigma_{1} $ through the set of fixed points of its initial self-embeddings.

\begin{lem}
	Suppose $ \mathcal{M}$ is a  nonstandard model of $ \mathrm{I}\Sigma_{1} $  in which $ \mathbb{N} $ is not strong. Then for any self-embedding  $ j $ of $ \M $ the following hold:
	\begin{itemize}
	\item[(1)] $ \fix(j) $ is 1-tall.
	\item[(2)] If $ \fix(j)$ is a countable model of $\mathrm{B}\Sigma_{1} $, then it is 1-extendable.
	\end{itemize} 
\end{lem}
\begin{proof}
\begin{itemize}
	
	
	\item[(1)] Let $ a\in\fix(j) $ be arbitrary and fixed. Since $ \fix(j)\prec_{\Sigma_{1}}\M $, it suffices to prove that $ \mathrm{K}^{1}(\M;a) $ is not cofinal in $ \fix(j) $. Since $ \M\models\mathrm{B}^{+}\Sigma_{1} $,  there exists some $ t_{0}\in\mm $ such that $\mathrm{K}^{1}(\M;a)<t_{0}  $.
Moreover, by Lemma 2.5(2) there exists some $ t_{00}\in\fix(j) $ such that  ${\mathrm{Th}_{\Sigma _{1}}(\mathcal{M};t_{0},a)\subseteq \mathrm{Th}_{\Sigma _{1}}(\mathcal{M};t_{00},a)}$. Therefore,   $ \mathrm{K}^{1}(\M;a)<t_{00} $.
 \item [(2)] By Theorem 2.2(2), and part (1) of this lemma, it suffices to prove that $ \mathbb{N} $ is not $ \Pi_{1}$-definable in $ \fix(j) $. Suppose not; i.e. $ \mathbb{N} $ is definable in $ \fix(j) $ by some $ \Pi_{1}$-formula $ \pi(x) $. By Lemma 2.5(1), $ \fix(j)\cap\mm\setminus\mathbb{N} $ is downward cofinal in $ \mm\setminus\mathbb{N} $. So by $ \Sigma_{1}$-Underspill in $ \M $, there exists some $ n\in\mathbb{N} $ such that $ \M\models\neg\pi(n) $, and consequently since $ \fix(j)\prec_{\Sigma_{1}}\M $, $ \fix(j)\models\neg\pi(n) $, which is a contradiction.
\end{itemize}	
\end{proof}

The following corollary  generalizes Theorem 1.2:
\begin{cor}
	Let $ \mathcal{M}\models\mathrm{I}\Sigma_{1} $ be countable and nonstandard in which $ \mathbb{N} $ is not strong,  and $ j  $ is an   initial self-embedding  of $ \mathcal{M} $ such that $ \mathrm{Fix}(j)\models\mathrm{B}\Sigma_{1} $. Then $ \fix(j) $ is isomorphic to a proper  cut of $ \mathcal{M} $.
\end{cor}
\begin{proof}
	By Theorem 2.2(1) and the previous lemma,  it is enough to prove that $ {\mathrm{SSy}(\mathrm{Fix}(j))=\mathrm{SSy}(\mathcal{M})} $. So let $  X=\mathbb{N}\cap a_{\mathrm{E}} $ for some $ a\in M $. Since $ \mathbb{N} $ is not strong in $ \mathcal{M} $, by Lemma 2.5(2) there exists some $ b\in\mathrm{Fix}(j) $ such that $ \mathrm{Th}_{\Sigma_{1}}(\mathcal{M};a)\subseteq\mathrm{Th}_{\Sigma_{1}}(\mathcal{M};b) $. Therefore, $  X=\mathbb{N}\cap b_{\mathrm{E}} $, and this finishes the proof.
\end{proof}

We conclude this section with  a generalization of a similar result about automorphisms of countable recursively saturated models of $ \pa $ in \cite{ks}. Moreover, the following corollary refines Theorem 2.4(3).
\begin{cor}
	Let $ \mathcal{M}\models\mathrm{I}\Sigma_{1} $ be countable and nonstandard. Then the following are equivalent:
	\begin{itemize}
		\item[1)] $ \mathbb{N} $ is strong in $ \mathcal{M} $.
		\item[2)] There exists some proper initial self-embedding $ j $ of $ \mathcal{M} $ such that $ \mathrm{Fix}(j)=\mathrm{K}^{1}(\mathcal{M}). $
		\item[3)] There exists some proper initial self-embedding $ j $ of $ \mathcal{M} $, and some small $ \mathcal{M}_{0}\prec_{\Sigma_{1}}\mathcal{M} $, such that $ \mathrm{Fix}(j)= M_{0}. $
		\item[4)] For every small $ \mathcal{M}_{0}\prec_{\Sigma_{1}}\mathcal{M} $ there exists some proper initial self-embedding $ j $ of $ \mathcal{M} $ such that $ \mathrm{Fix}(j)= M_{0}. $
		\item[5)] There exists some proper initial self-embedding $ j $ of $ \mathcal{M} $ such that $ \mathrm{Fix}(j)\subseteq\mathrm{I}^{1}(\mathcal{M}) $.
		\end{itemize}
	If $ \M\models \pa $ and it is  recursively saturated, then the above statements are equivalent to the following:
	\begin{itemize}
		\item[6)] There exists some proper initial self-embedding $ j $ of $ \mathcal{M} $ such that $ \mathrm{Fix}(j)\models\mathrm{B}\Sigma_{1} $ and it  is  isomorphic to no proper initial segments of $ \mathcal{M} $.
	\end{itemize} 
\end{cor}
\begin{proof} The equivalences  of statements (1) to (5)  is 
	a straightforward implication of Corollary 3.5 and Theorem 4.1(1). Moreover, $ (6)\Rightarrow(1) $ holds by Corollary 4.2. In order to prove $ (4)\Rightarrow(6) $, similar to the proof of  Theorem 3.2(3), we will find some small recursively saturated elementary submodel $ \M_{0} $ of $ \M $. So statement (4) will provide us with a proper initial self-embedding $ j $ of $ \M $ such that $ \fix(j)=M_{0} $. Clearly $ \fix(j)\models\mathrm{B}\Sigma_{1} $. Moreover,  as we mentioned in the beginning of  Section 3, $ \ssy(\M_{0})\neq\ssy(\M) $. As a result, $ \fix(j) $ is isomorphic to no proper initial segment of $ \M $. 
\end{proof}

\section{Extendability}
In this section, we will study the extendability of initial embeddings. Most of the theorems of this section are generalizations of results about automorphisms of countable recursively saturated models of $ \pa $ obtained in \cite{kos-kot} and \cite{kos-kot96}. 
\begin{dfn}
	Suppose $ \mathcal{M}$ and $ \mathcal{N} $ are models of $\mathrm{I}\Sigma_{1} $,  $ \mathcal{M}_{0}$ and $\mathcal{N}_{0}$ are bounded  submodels (or proper cuts) of $\mathcal{M} $ and $ \mathcal{N} $ respectively.
We call an  initial embedding $ j:\mathcal{M}_{0}\hookrightarrow\mathcal{N}_{0} $  an\textit{ initial $ (\mathcal{M},\mathcal{N})$-embedding}  if for every $  A\subseteq M_{0} $ it holds that:
	\begin{center}
		$  A\in\ssy_{I}(\M) $  iff $ j( A) \in\ssy_{J}(\N)$,\\
		where  $ I:=\I^{1}(\M;\mm_{0}) $, and $ J:=\I^{1}(\N;j(M_{0})) $.
	\end{center}
	If $ \mathcal{M}=\mathcal{N} $, we call such $ j $ an \textit{initial
	 $ \mathcal{M} $-embedding}. \\

\end{dfn}
First, in the next lemma we will show that the condition in the above definition, i.e. preserving coded subsets, is a  necessary condition for extendability of an initial embedding:
\begin{lem}
	Suppose $ \mathcal{M}$ and $ \mathcal{N} $ are models of  $\mathrm{I}\Sigma_{1} $, $ \mathcal{M}_{0}\subseteq\mathcal{M}$  and $\mathcal{N}_{0}\subseteq\mathcal{N}$ are bounded   submodels (or proper cuts), and $ j:\mathcal{M}_{0}\hookrightarrow\mathcal{N}_{0} $ is an initial embedding. If $ j $ is extendable to some initial embedding $ \hat{j}:\mathcal{M}\hookrightarrow\mathcal{N} $, then   $ j $ is an initial $ (\mathcal{M},\mathcal{N})$-embedding.
\end{lem}
\begin{proof}
Put $ I:=\I^{1}(\M;\mm_{0}) $,  $ J:=\I^{1}(\N;j(M_{0})) $, and let $  A\subseteq M_{0}  $ be arbitrary. If $  {A=I\cap(\alpha_{\E})^{\M}} $ for some $ \alpha $ in $ \mathcal{M} $, then clearly $ j(A)=J\cap ((\hat{j}(\alpha))_{\E})^{\N}$. Conversely, suppose  $ j( A)\in\ssy_{J}(\N) $. Since $ \mm_{0} $ is bounded in $ \M $, we have $ J\subsetneq_{e}\hat{j}( M) $.   As a result, $ j(\A)\in\ssy_{J}(\hat{j}( M)) $, which implies that $A\in\ssy_{I}(\M)  $.
\end{proof}

  Converse of the above lemma holds, when $ \M_{0} $ and $ j(\mm_{0}) $ are $ \Sigma_{1} $-elementary initial segments of  $ \M $ and $ \N $:
\begin{theorem}
	Suppose $ \mathcal{M}$ and $ \mathcal{N} $ are countable and nonstandard models of $\mathrm{I}\Sigma_{1} $, and $  I$ and $ J $  are $ \Sigma_{1} $-elementary initial segments of  $ \M $ and $ \N $, respectively. Then for any  isomorphism $ {j: I\rightarrow J} $ which is an initial  $ (\mathcal{M},\mathcal{N})$-embedding and each $ b>J $, there exists some  proper initial embedding  $ \hat{j}:\mathcal{M}\hookrightarrow\mathcal{N} $ such that $ \hat{j}\upharpoonright_{I}=j $ and $ \hat{j}(M)<b $. 
\end{theorem}
\begin{proof}[Sketch of proof] The proof is conducted by a  back-and-forth argument similar to the one used in the proof of \cite[Thm. 3.3]{our}; we will build   finite partial functions $ \bar{u}\mapsto\bar{v} $ such that  the following  induction hypothesis holds:
	\begin{center}
	If 	$ \mathcal{M}\models  [f(\bar{u},i)\downarrow]$, then $\N\models  [f(\bar{v},j(i))\downarrow] ^{<b}$,\\
		for every $ f\in\F$ and $ i\in I $.
	\end{center}
 For the `forth' steps, if $ \bar{u}\mapsto\bar{v} $ is built, for given $ m\in M $ we define:\\

	$ H:=\lbrace\langle r,i\rangle\in I: \ \mathcal{M}\models [f_{r}(\bar{u},m,i)\downarrow] \rbrace $.\\
	
Then, let $ h\in M $ such that $ H=I\cap h_{\E} $. Since $ j $ is an initial $ (\M,\N) $-embedding,  there exists some $ h'\in N $ such that $ j(H)=J\cap h'_{\E} $.  Therefore, by induction hypothesis for every $ s\in I $ it holds that:\\

$ (1): \ \ $	$ \mathcal{N}\models\exists x, w<b \ \forall \langle r,i\rangle <j(s) \ (\langle r,i\rangle\mathrm{E}h'\rightarrow [f_{r}(\bar{v},x,i))\downarrow]^{<w} ) $.\\
	
	Since $ j $ is onto, statement (1) implies that for every $ t\in J $ it hold that:\\

$ (2): \ \ $	$ \mathcal{N}\models\exists x, w<b \ \forall \langle r,i\rangle <t \ (\langle r,i\rangle\mathrm{E}h^{'}\rightarrow[f_{r}(\bar{v},x,i))\downarrow]^{<w})  $.\\

Therefore, by using  $ \Sigma_{1}$-Overspill in $ \mathcal{N} $, we will find some image for $ m $, for which induction hypothesis holds. The `back' stages can be done similarly.  
\end{proof}

The proof of the  above theorem can also be modified for  $ I$-small submodels:

\begin{theorem}
	Suppose $ \mathcal{M}\models\mathrm{I}\Sigma_{1} $ is countable and nonstandard, $ I $ is a strong cut of $ \M $,   $ \mathcal{M}_{0} $ is an $  I$-small $ \Sigma_{1} $-elementary  submodel of $ \mathcal{M} $ such that $  M_{0}:=\{(a)_{i}: i\in I\} $, and $ j $ is an initial embedding of $ \mathcal{M}_{0} $ such that $ j(I)\subseteq_{e}\M $.
	 Then the following are equivalent: 
	\begin{itemize}
	\item[(1)] $ j \upharpoonright_{I}$ is an initial $ \M$-embedding, and there exists some $ b\in M $ such that $ 
	{\mathcal{M}\models j((a)_{i})=(b)_{j(i)}} $ for all $ i\in I $. 
	\item[(2)]  $ j $ extends to some proper initial self-embedding of $ \mathcal{M} $.
	\end{itemize}
	\end{theorem}
\begin{proof}[Sketch of proof]
$(2) \Rightarrow(1) $ holds by Lemma 5.1. In order to prove $ (1)\Rightarrow(2) $, we will use  a similar argument to the proof of \cite[Thm. 3.3]{our} to obtain  an extension $ \hat{j} $ of $ j $. For this purpose, first we will fix some $ d\in M $ which  is an upper bound for $ \mm_{0} $. Then, we will  build   finite partial functions $ \bar{u}\mapsto\bar{v} $ such that  the following  induction hypothesis holds:
\begin{center}
 	$ \mathcal{M}\models[f(\bar{u},(a)_{i})\downarrow]\rightarrow[f(\bar{v},(b)_{j(i)})\downarrow]^{<d} $,\\
	for every $ f\in\F $ and $ i\in I $.
\end{center}
Here, we outline the proof for the `back' steps and the proof of `forth' steps is left to the reader. Suppose $ \bar{u}\mapsto\bar{v} $ is built, and $ m<\max\{\bar{v}\} $ is given. We define:\\

$ L:=\lbrace\langle r,i\rangle\in j(I): \ \mathcal{M}\models  \neg[f_{r}(\bar{v},m,(b)_{i})]^{<d} \rbrace $.\\

Then, let $ l\in M $ such that $ L=j(I)\cap l_{\E} $. Since $ j\upharpoonright_{I} $ is an initial $ \M $-embedding, then there exists some $ l'\in M $ such that $ j^{-1}(L)=I\cap l'_{\E} $. Moreover, by using Lemma 3.3, for every $ s\in I $ there exists some $ (a)_{i_{s}}\in M_{0} $ which  codes of the following subset of $ M_{0} $:

$$A:=\{\langle r,(a)_{i}\rangle: \ \M\models(\langle r,i\rangle <s \ \wedge \langle r,i\rangle\E l' ) \} .$$

By $ \Pi_{1} $-Overspill, it suffices to prove that  for every $ s\in I $ it holds that:\\

$ (\star): \ \ $	$ \M\models\exists x<\max\{\bar{u}\} \ \forall \langle r,i\rangle<s  \ (\langle r,i\rangle\E l'\rightarrow \neg[f_{r}(\bar{u},x,(a)_{i})\downarrow] ) $.\\

Suppose not; i.e. there exists some $ s\in I $ which for statement $ (\star) $ does not hold. So we have:\\

$ (i): \ \ $	$ \M\models\forall x<\max\{\bar{u}\} \ \exists \langle r,i\rangle<s  \ (\langle r,i\rangle\E l' \ \wedge [f_{r}(\bar{u},x,(a)_{i})\downarrow] ) $.\\

As a result, by using $ \Sigma_{1} $-Collection in $ \M $, from statement $ (i) $, induction hypothesis, and the way we chose $(a)_{i_{s}} $, we may conclude that:\\

$ (ii): \ \ $	$ \M\models\forall x<\max\{\bar{v}\} \ \exists \langle r,\epsilon\rangle\E (b)_{j(i_{s})} \ ( [f_{r}(\bar{v},x,\epsilon)\downarrow]^{<d} ) $.\\

So by statement $ (ii) $, there exists some $ \langle r,i\rangle<s $ such that:\\

$ (iii): \ \ $ $ \M\models (\langle r,i\rangle\E l' \ \wedge \ [f_{j(r)}(\bar{v},m,(b)_{j(i)})\downarrow]^{<d} ) $. \\

But statement $ (iii) $ is in direct contradiction with the way we chose $ l' $.
 
\end{proof}
In the last theorem, we  investigate  whether we can control the set of fixed points, while extending an isomorphism to an initial self-embeddings with larger domain:
\begin{theorem}
Suppose $ \mathcal{M}\models\mathrm{I}\Sigma_{1} $ is countable and nonstandard,  $ I $ is a strong $ \Sigma_{1} $-elementary  initial segment of $ \mathcal{M} $, and $ j: I\rightarrow I $ is an isomorphism and an initial $ \mathcal{M}$-embedding. Then there exists some proper initial self-embedding  $ \hat{j} $ of $\mathcal{M}$ such that $ \hat{j}\upharpoonright_{ I}=j $, and $ {\mathrm{Fix}(\hat{j})=\mathrm{Fix}(j)} $. 
\end{theorem}
\begin{proof}[Sketch of proof]	
First, we will fix some arbitrary $ a>I $. Since  $ I $ is strong in $ \M $, there exists some $ b>I $ such that:
		\begin{center}
	$ (\star): \ \ $	if $ \M\models [f(a,i)\downarrow] $ and $ f(a,i)>I $ then $ f(a,i)>b $, for all $ f\in\F $ and $ i\in I $.
		\end{center}
		 So by Theorem 5.2, there exists some proper initial self-embedding $ \bar{j} $   of $ \M $ such that $ \bar{j}\upharpoonright_{I}=j $ and $ \bar{j}(M)<b $. If $ \fix(\bar{j})=\fix(j) $, then we are done. Otherwise we will build $ \hat{j} $ in the following way:
		\begin{itemize}
	\item By using a similar argument to the proof of Theorem 3.4, we construct some proper initial self-embedding $ h $ of $ \N:=\mathrm{H}^{1}(\M;a) $ such that $ h\upharpoonright_{I}=j $, $ \fix(h)=\fix(j) $, and $ h(\nn)<b $. 	
	In order to construct such $ h $, we will inductively construct finite functions $ \bar{u}\mapsto\bar{v} $  such that:

				$ \mathrm{P}(\bar{u},\bar{v})\equiv  [f(\bar{u},i)\downarrow]\rightarrow[f(\bar{v},j(i))\downarrow]^{<b} $, for all $f\in\mathcal{F} $ and $ i\in I $; and\\
				$ \mathrm{Q}(\bar{u},\bar{v})\equiv \left(\begin{array}{c}
				[f(\bar{u},i)\downarrow] \wedge [f(\bar{v},j(i))\downarrow]^{<b}\wedge \\ f(\bar{u},i)\notin  I 
				\end{array}\right)  \Rightarrow f(\bar{u},i)\neq f(\bar{v},j(i)) $, for all $ f\in\mathcal{F}$ and all $ i\in I $.
		\begin{itemize}	
			\item For the first step of induction, we will take $ a\mapsto \bar{j}(a) $; clearly $ \mathrm{P}(a,\bar{j}(a)) $ holds in $ \M $. Moreover, by statement $ (\star) $ and since $ \fix(\bar{j})<b $, the property $ \mathrm{Q}(a,\bar{j}(a)) $  also holds in $ \M $. 
			\item Then suppose $ \bar{u}\mapsto\bar{v} $ is built. We will just mention the changes that should be made in the `forth' steps of Theorem 3.4, and `back' steps should be modified similarly:
		\begin{itemize}
				\item Suppose $ m\in N\setminus\{\bar{u}\} $ is given. By the definition of $ \N $, without loss of generality, we may assume that $ m\leq t(\bar{u},a) $ for some $ t\in\F $. Put: 
		$$ C:= \{\langle r,i\rangle\in I: \N\models[f_{r}(\bar{u},m,i)\downarrow]\wedge f_{r}(\bar{u},m,i)\notin\mathrm{K}^{1}(\N;I\cup\{\bar{u}\})  \} .$$
		
		Let $  \alpha,\alpha'\in N $ such that $  C=I\cap \alpha_{\E} $ and $ j( C)=I\cap\alpha'_{\E} $ (note that since $ \N $ is a $ \Sigma_{1} $-elementary initial segment of $ \M $ containing $ I $, $ j $ is an initial $ \N$-embedding):
			\item Let  $L:= \{\langle r,i\rangle\in I: \N\models[f_{r}(\bar{u},m,i)\downarrow]\} $, $ L=I\cap\beta_{\E} $, and  $ j(L)=I\cap\beta'_{\E} $ for $ \beta,\beta'\in N $.
				\item For every $ s\in\bar{j}(N) $ such that $ \bar{j}(s')=s $ for some $ s'\in N $, let:\\ 	$ p_{s}(y):= \lbrace y< t(\bar{v},\bar{j}(a))\rbrace\cup p_{s1}(y)\cup p_{s2}(y)$; where: \\
			${p_{s1}(y):=\lbrace\forall i<s(\langle n,i\rangle\E\beta'\rightarrow [f_{n}(\bar{v},y,i)\downarrow]^{<b}): \ n\in\mathbb{N}\rbrace;} $ and \\ 
				$p_{s2}(y):=\left\lbrace \forall w<s'\ \forall i<s  
			\left(\begin{array}{c}
			\left(\begin{array}{c}\langle n,w\rangle\E\alpha \ \wedge \ \langle n,i\rangle\mathrm{E}\alpha'\ \wedge \\
			
			 [f_{n}(\bar{v},y,i)\downarrow]^{<b}\end{array}\right) \rightarrow\\ 
		f_{n}(\bar{u},m,w)\neq f_{n}(\bar{v},y,i)
			\end{array}\right) :   \ n\in\mathbb{N} \right\rbrace$.
				\item In order to  find some $ s>I $ such that $ s\in \bar{j}(N) $ and $ p_{s}(y) $ is finitely satisfiable, we will adapt  the rest of the proof of Theorem 3.4  accordingly; for instance, we will mention two of these adaptations: 
				
					\item[$ (1) $] Let $ d'\in N $ such that $ d'>I $ and $ d:=\bar{j}(d') $. Moreover, for every $i,s,s'\in N$, let $  \Theta(s,s',i,\bar{u},m,\bar{v},b,\bar{j}(a),\alpha,\alpha',\beta') $ be the following $ \Delta_{0}$-formula:\\
				
						${\forall r<i \ \exists y\leq t(\bar{v},\bar{j}(a))  \left(\begin{array}{c}
							\forall w<s  (\langle r,w\rangle\mathrm{E}\beta'\rightarrow [f_{r}(\bar{v},y,w)\downarrow]^{<b}) \ \wedge  
							\\ \forall w<s\forall w'<s' \forall r'<i\left(\begin{array}{c}
							\left(\begin{array}{c}\langle r',w'\rangle\mathrm{E}\alpha'\wedge\\	\langle r',w\rangle\mathrm{E}\alpha\wedge\\
						 
						 	[f_{r'}(\bar{v},y,w)\downarrow]^{<b}\end{array}\right)\rightarrow \\ f_{r'}(\bar{u},m,w')\neq f_{r'}(\bar{v},y,w)
							\end{array}\right)
							\end{array}\right) } $.\\
				
				Then, for every $ i\in M $, we define:\\
				
						$ g(i):=\mathrm{max}\lbrace  w<d': \ \mathcal{M}\models\exists x\leq d \ \Theta(x,w,i,\bar{u},m,\bar{v},b,\bar{j}(a),\alpha,\alpha',\beta') \rbrace $.\\

				Since $  I $ is strong, there exists some $ e'>  I $ such that $ e'\leq d' $,  and for all $ i\in  I $, $ g(i)>  I $ iff $ g(i)>e' $. Then, for every $ i\in M $ put:\\
				
					$ l(i):=\mathrm{max}\left\lbrace  x<\bar{j}(e'): \ \mathcal{M}\models \left(\begin{array}{c}
					[g(i)\downarrow]^{<d'} \ \wedge \ g(i)>e'\ \wedge  \\ \Theta(x,g(i),i,\bar{u},m,\bar{v},b,\bar{j}(a),\alpha,\alpha',\beta')
					\end{array}\right) \right\rbrace $.\\
					
					 Again, since $  I $ is strong, there exists some $ e>  I $ such that $ e\leq d $,  and for all $ i\in  I $, $ l(i)>  I $ iff $ l(i)>e $. Then $ p_{e}(y) $ is a finitely satisfiable type.
				\item[$ (2) $] Instead of the function $ \langle o(\diamondsuit,y),h(\diamondsuit,y)\rangle $ we need to define the following function:\\
				
				$ \langle o(\diamondsuit,y),h(\diamondsuit,y),h'(\diamondsuit,y)\rangle:=\min\left\lbrace\langle n_{t},i,w\rangle\E\alpha_{s_{_{0}}}: 
				\left(\begin{array}{c}
				[b(\diamondsuit,y)\downarrow]\wedge\\
				
				[f_{n_{t}}(\diamondsuit,y,i)\downarrow]^{<b(\diamondsuit,y)} \ \wedge\\
				
				[s_{t}(\diamondsuit,w)\downarrow]^{<b(\diamondsuit,y)} \wedge\\ s_{t}(\diamondsuit,w)= f_{n_{t}}(\diamondsuit,y,i)\end{array}\right) 
				\right\rbrace $;\\ 
				where $ \alpha_{s_{0}}\in I $ is the code of the following subset of $ I $:
				$$ \{\langle n,i,w\rangle: \ \M\models i<s_{0}\wedge w<j^{-1}(s_{0})\wedge\langle n,w\rangle\E\alpha\wedge\langle n,i\rangle\E\alpha'\}. $$

				\end{itemize}  
	The rest of the adaptations should be made similar to statements (1) and (2) in order to construct $ h $.
		
	\item 	If $ \M=\N $, then we are done. Otherwise, by using Theorem 2.3 we shall find some proper initial embedding $ k:\M\hookrightarrow\N $   such that $ I\subseteq\I_{\mathrm{fix}}(k) $ and $ b\in k(\mm) $.
		\item Finally, we put $ \hat{j}:=k^{-1}hk $.
	
\end{itemize}
\end{itemize} 
\end{proof}
\begin{rem}
If we let $ j $ be the trivial automorphism of $ I $,  then Theorem 5.4 implies Theorem 2.4(2).
\end{rem}


\begin{thebibliography}{99}
\bibitem{our} S.~Bahrami and A.~Enayat, \textit{Fixed points of
	self-embeddings of models of arithmetic}, \textbf{Ann.~Pure Appl.~Logic} 169, 2018, pp.~487-513.

\bibitem{bs} J. Barwise and J. Schlipf,\textit{ On recursively saturated models of arithmetic},
 \textbf{Model Theory and Algebra: a memorial tribute
to A. Robinson} (edited by D. Saracino and V. Weispfenning),
Springer Lecture Notes in Mathematics, vol. 498, 1976, pp. 42-55.


\bibitem{dp} c. Dimitracopoulos, J. Paris,  \textit{A note on a theorem of H.
Friedman}, \textbf{Z. Math. Logik Grundlag. Math.}, 34(1), 1988, pp.  13–17.

\bibitem{en} A. Enayat, \textit{Automorphisms of models
	of arithmetic: a unified view, }\textbf{Ann.~Pure Appl.~Logic} 145 (2007), pp.~16-36.


\bibitem{fr} H.~Friedman, \textit{Countable models of set theories},
\textbf{Lecture Notes in Math}.~337, Springer, Berlin, 1973, pp.~539-573.

\bibitem{hp} P.~H\'{a}jek and P.~Pudl\'{a}k, \textbf{Metamathematics of First Order Arithmetic}, Springer,
Heidelberg, 1993.

\bibitem{kaye} R.~Kaye, \textbf{Models of Peano Arithmetic}, Oxford
University Press, Oxford, 1991.

\bibitem{kkk} R.~Kaye, R.~Kossak, and H.~Kotlarski, \textit{Automorphisms of recursively saturated models of arithmetic}, \textbf{Ann.~Pure Appl.~Logic}~55 1991, pp.~67-99.

\bibitem{kp}  L.~Kirby and J.~Paris, $\Sigma _{n}$-\textit{Collection schemas in arithmetic}, in \textbf{Logic Colloquium '77,}
North-Holland Publishing Company, Amsterdam, 1978, pp.~199-209.


\bibitem{kos3} R. Kossak, \textit{A note on satisfaction classes}, \textbf{Notre Dame of Formal Logic}, vol.
26, 1985, pp. 1-8.


\bibitem{kos-kot} R.~Kossak, and H.~Kotlarski, \textit{Results on automorphisms of recursively saturated models of $ \mathrm{PA} $}. \textbf{Fund. Math.}, 129(1), 1988, pp. 9–15.

\bibitem{kos-kot96} R.~Kossak, and H.~Kotlarski,  \textit{On extending automorphisms	of models of Peano arithmetic}, \textbf{Fund. Math.}, 149(3), 1996, pp. 245–263. 

\bibitem{ks} R.~Kossak and J.~Schmerl, \textit{Arithmetically saturated models of arithmetic}, \textbf{Notre Dame J. Formal Logic} 36(4), 1995, pp. 531–546


\bibitem{ksb} R.~Kossak and J.~Schmerl,, \textbf{The Structure of Models of Peano Arithmetic}, Oxford, 2006.

\bibitem{las} D. Lascar, \textit{The small index property and recursively saturated models of Peano arithmetic}, \textbf{Automorphisms	of first-order structures},  Oxford University Press, New York, 1994 pp. 281–292.



\bibitem{w}  A. Wilkie, \textit{On the theories of end-extensions of models of arithmetic}, \textbf{Lecture Notes in Mathematics}, 619, 1977, pp. 305–310. 

\end{thebibliography}
\end{document}